\documentclass[11pt]{amsart}
\usepackage{graphicx}
\usepackage{amsmath,amsxtra,amssymb,latexsym, amscd,amsthm}

\usepackage{floatflt,lipsum}

\theoremstyle{plain}
\newtheorem{theorem}{Theorem}[section]
\newtheorem{corollary}[theorem]{Corollary}
\newtheorem{lemma}[theorem]{Lemma}

\newtheorem{conjecture}[theorem]{Conjecture}
\theoremstyle{definition}

\theoremstyle{remark}

\numberwithin{equation}{section}
\numberwithin{figure}{section}
\numberwithin{table}{section}

\newcommand{\M}{\operatorname{M}}

\begin{document}
\setlength{\baselineskip}{16truept}
\title{Proof of Blum's Conjecture on Hexagonal Dungeons}
\maketitle

\begin{center}
MIHAI CIUCU\footnote{Department of Mathematics, Indiana University, Bloomington IN 47405,\\ email: mciucu@indiana.edu}
\and TRI LAI\footnote{Corresponding author -- Department of Mathematics, Indiana University, Bloomington IN 47405, email: tmlai@indiana.edu, tel: 812-855-2263}
\end{center}


\begin{abstract}
Matt Blum conjectured that the number of tilings of the  Hexagonal Dungeon of sides $a,\ 2a,\ b,\  a,\ 2a,\ b$ (where $b\geq 2a$) is $13^{2a^2}14^{\lfloor\frac{a^2}{2}\rfloor}$ (J. Propp, \textit{New Perspectives in Geometric Combinatorics}, Cambridge University Press, 1999). In this paper we present a proof for this conjecture using Kuo's Graphical Condensation Theorem (E. Kuo, \textit{Applications of Graphical Condensation for Enumerating Matchings and Tilings}, Theoretical Computer Science, 2004).
\end{abstract}

\section{Introduction}

In 1999 Propp published an article \cite{Propp} tracking the progress on a list of 20 open problems in the field of exact enumeration of perfect matchings, which he presented in a lecture in 1996, as part of the special program on algebraic combinatorics organized at MSRI during the academic year 1996--1997. The article also presented a list of 12 new open problems.

These 32 problems can be grouped into three broad categories: conjectures stating an explicit formula for the number of perfect matchings of the specific family of graphs they pertain to, problems for which the number of perfect matchings does not seem to be given by a simple formula, but presents some patterns that are required to be proved, and problems concerned with various aspects of the Kasteleyn matrices of the involved graphs, and not directly with their number of perfect matchings.

In some sense, the most compelling ones to prove are the ones in the first category. The only one from the list of 12 new problems which falls into this category is Matt Blum's conjecture on the number of tilings of the so-called hexagonal dungeon regions\footnote{We note that there is one problem among the original 20  (Problem 16 in \cite{Propp}) which stands out in a similar manner. This was solved and generalized recently by one of the authors of the current paper (T.L.); see \cite{hybrid}.}. Proving this conjecture, still open fourteen years after its publication, is the main result of the current paper.




%

Consider the lattice obtained from the triangular lattice by drawing in all the altitudes in all the unit triangles (i.e. the plane lattice corresponding to the affine Coxeter group $G_2$). On this lattice, consider a hexagonal contour of the type illustrated in Figure \ref{hexagon}. If the side-lengths of the hexagonal contour, in units equal to the side-length of the unit triangles, are $a$, $2a$, $b$, $a$, $2a$, $b$ (in clockwise order, starting from the western edge), then the lattice region determined by the indicated jagged contour is called the {\it hexagonal dungeon of sides $a$, $2a$, $b$, $a$, $2a$, $b$}, and is denoted $HD_{a,2a,b}$. This region was introduced by Matt Blum, who discovered a striking pattern in the number of its tilings\footnote{A tile is a union of two fundamental regions sharing an edge, and a tiling of a lattice region $R$ is a covering of $R$ by such tiles, with no gaps or overlaps.}, which led him to the following conjecture.


\begin{conjecture} [Matt Blum, Problem 25 in \cite{Propp}]\label{MBconjecture}
Assume that $a$ and $b$ are two positive integers so that $b\geq 2a$.  Then the number of tilings of the hexagonal dungeon
$HD_{a,2a,b}$
is $13^{2a^2}14^{\lfloor\frac{a^2}{2}\rfloor}$.
\end{conjecture}


The main result of the current paper is a proof of this conjecture. Our proof is based on Kuo's powerful graphical condensation method \cite{Kuo}.
In order for graphical condensation to work, we need to extend the original conjecture to a more general family of regions. These more general regions, as well as the corresponding extension of Matt Blum's conjecture, are presented in Section 3. Section 4 shows how to obtain recurrences for the number of tilings of these regions, using the graphical condensation method. In Section 5, these recurrences are shown to be satisfied by the explicit formulas presented in Section 3. The proof of our extension of Blum's conjecture is presented in Section 6.

\begin{figure}\centering
\includegraphics[width=12cm]{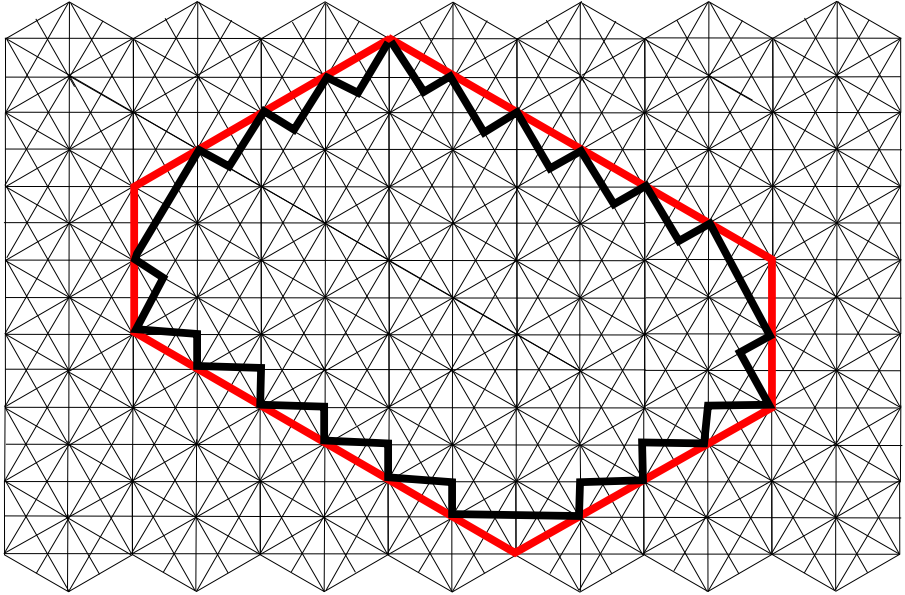}
\caption{The hexagonal dungeon of sides $2,$ $4,$ $6,$ $2,$ $4,$ $6$ (in cyclic order, starting from the western side).}
\label{hexagon}
\end{figure}

\section{Preliminaries}

A \textit{perfect matching} of a graph $G$ is a collection of edges with the property that each vertex of $G$ is incident to precisely one edge in the collection. Denote by $\M(G)$ the number of perfect matchings of a graph $G$. A \textit{forced edge} of a graph $G$ is an edge contained in every perfect matching of $G$. Therefore, removing forced edges does not change the number of perfect matchings of a graph.

Given a lattice in the plane, a (lattice) \textit{region} is a finite connected union of fundamental regions of that lattice. A \textit{tile} is the union of any two fundamental regions sharing an edge. A \textit{tiling} of the region $R$ is a covering of $R$ by tiles with no gaps or overlaps. The tilings of a region $R$ can be naturally identified with the perfect matchings of its dual graph (i.e., the graph whose vertices are the fundamental regions of $R$, and whose edges connect two vertices precisely when they correspond to fundamental regions that share an edge). In particular, the number of tilings of the region $R$ and the number of perfect matchings of its dual graph are equal. In view of this bijection, we denote the number of tilings of the lattice region $R$ by $\M(R)$.

An \textit{induced subgraph} of a graph $G$ is a graph whose vertex set is a subset $U$ of the vertex set of $G$, and whose edge set consists of all edges of $G$ with endpoints in $U$.  

\begin{lemma}[Graph Splitting Lemma]\label{graphsplitting}
Let $G$ be a bipartite graph, and let $V_1$ and $V_2$ be the two vertex classes.

Assume that an induced subgraph $H$ of $G$ satisfies following two conditions:
\begin{enumerate}
\item[(i)] \text{\rm{(Separating Condition)}} There are no edges of $G$ connecting a vertex in \\$V(H)\cap V_1$ and a vertex in $V(G-H)$.

\item[(ii)] \text{\rm{(Balancing Condition)}} $|V(H)\cap V_1|=|V(H)\cap V_2|$.
\end{enumerate}
Then
\begin{equation}\label{GS}
\M(G)=\M(H)\, \M(G-H).
\end{equation}
\end{lemma}

\begin{proof}
Color all vertices of $V_1$ white, and all vertices of $V_2$ black. We partition $\mu$ into three disjoint submatchings: $\mu_1$ consists of edges in $E(H)$, $\mu_2$ consists of edges in $E(G-H)$, and $\mu_3$ consists of edges connecting a vertex in $H$ and a vertex in $G-H$. To prove the lemma, it suffices to show that $\mu_3=\emptyset$.

Suppose otherwise that $\mu_3\not=\emptyset$. One can partition the vertex $V(G)$ of $G$ into three disjoint sets $S_1,S_2,S_3$, where $S_i$ is the set of vertices incident edges in $\mu_i$. Since $\mu_3\not= \emptyset$, we have $S_3$ contains at least one vertex of $H$.

Since each edge in $\mu_1$ connects a black vertex in $H$ and a white vertex of $H$, and each edge in $\mu_2$ is not incident to any vertex of $H$; the balancing condition implies that the numbers of black and white vertices of $H$ in $S_3$ are equal. However, by separating condition, $S_3$ does not contain any white vertices of $H$. It implies that $S_3$ does not contain any vertices of $H$, a contradiction the fact in the previous paragraph.
\end{proof}

As we mentioned in the introduction, our proof of Blum's conjecture is based on Kuo's graphical condensation method \cite{Kuo}. For completeness, we include here the variant that we will need.

\begin{theorem} [Kuo \cite{Kuo}] \label{kuothm} Let $G$ be a planar bipartite graph, and let $V_1$ and $V_2$ be the two vertex classes. Assume that $|V_1|=|V_2|$. Let $x,y,z$ and $t$ be four vertices appear in a cyclic order on a face of $G$. Assume in addition that $x,z\in V_1$ and $y,t\in V_2$. Then
\begin{equation}\begin{split}
\M(G)\M(G-\{x,y,z,t\})=\M(G-\{x,y\})\M(G-\{z,t\})\\+\M(G-\{t,x\})\M(G-\{y,z\}).
\end{split}
\end{equation}
\end{theorem}

\section{Extension of the conjecture}

\begin{figure}\centering
\includegraphics[width=12cm]{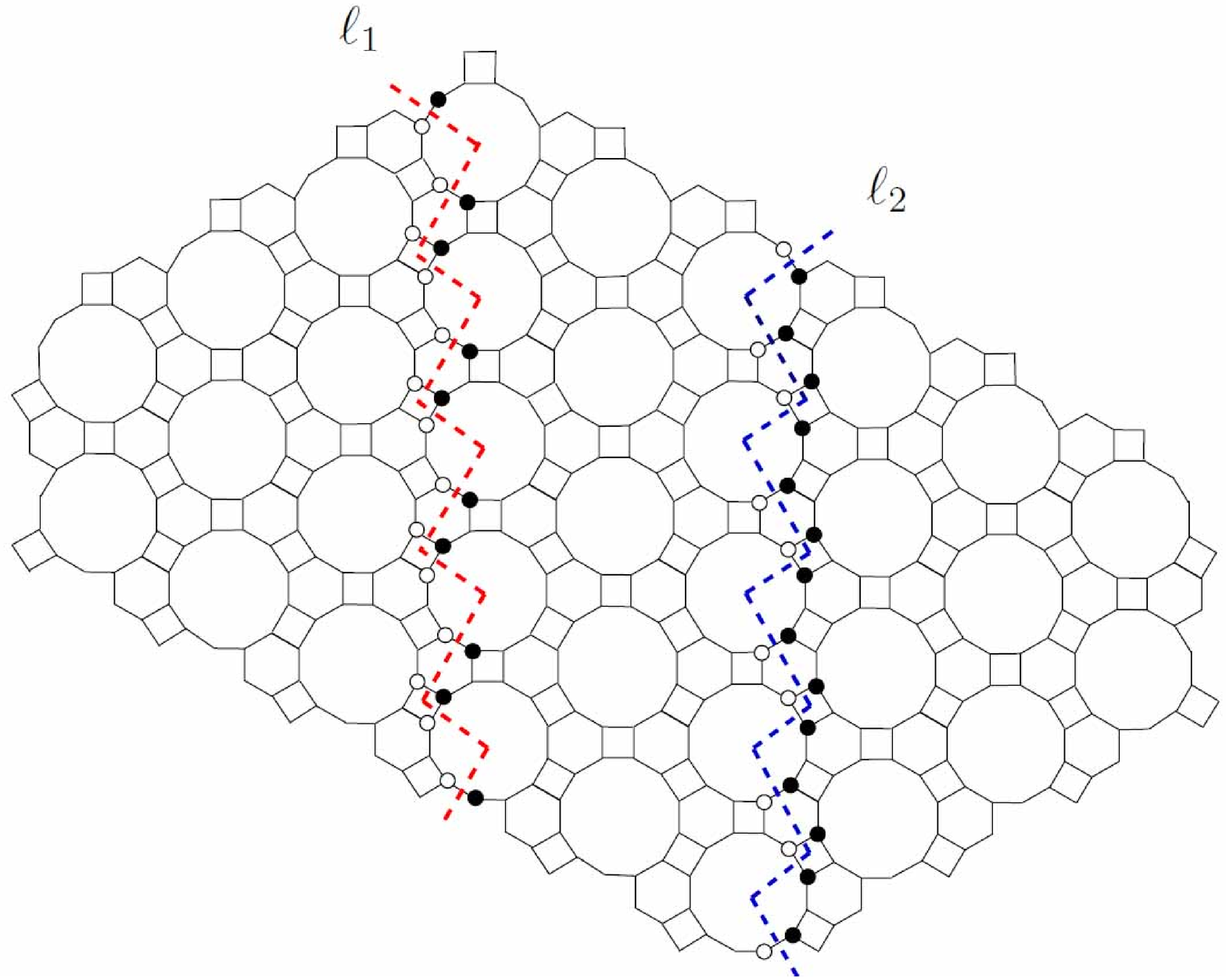}
\caption{The dual graph of the hexagonal dungeon of sides 2,4,6,2,4,6, and two zigzag cuts $\ell_1$ and $\ell_2$.}
\label{dual}
\end{figure}
\begin{figure}\centering
\includegraphics[width=12cm]{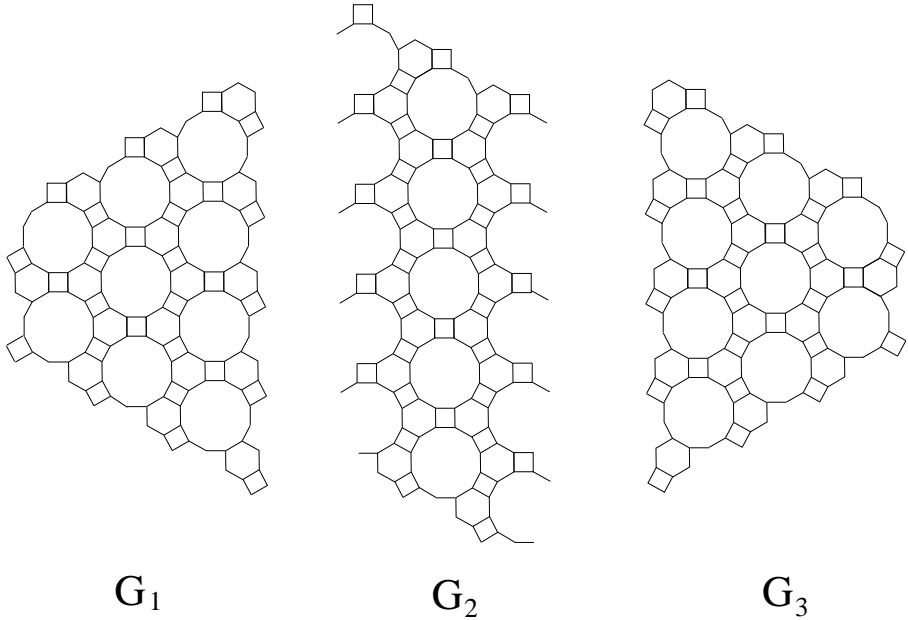}
\caption{Dividing the dual graph of the hexagonal dungeon of sides 2,4,6,2,4,6 into three parts.}
\label{split}
\end{figure}

Suppose $b\geq2a$, and let $G$ be the dual graph of the hexagonal dungeon of sides $a, 2a, b, a, 2a ,b$. Consider two zigzag cuts $\ell_1$ and $\ell_2$ on $G$ as indicated in Figure \ref{dual} (the illustrated case corresponds to $a=2$ and $b=6$).

The cuts divide the graph into three connected components. Denote them, from left to right, by $G_1$, $G_2$ and $G_3$ (see Figure \ref{split}).

$G$ is clearly a bipartite graph, and one readily checks that $G$ and its induced subgraph $G_1$ satisfy the assumptions of Lemma \ref{GS}. In a similar fashion, $G-G_1$ and its induced subgraph $G_3$ also satisfy the assumptions of Lemma \ref{GS}. We obtain therefore that
\begin{equation}\label{refinement1}
\M(G)=\M(G_1)\M(G-G_1)=\M(G_1) \M(G_2) \M(G_3).
\end{equation}

Since $G_1$ and $G_3$ are isomorphic, we have
\begin{equation}\label{refinement2}
\M(G_1)=\M(G_3).
\end{equation}

Moreover, by considering 
forced edges, one readily sees that the graph $G_2$ has a unique perfect matching, so $\M(G_2)=1$. Thus, from (\ref{refinement1}) and  (\ref{refinement2}) we obtain
\begin{equation}\label{refinement3}
\M(G)=\M(G_1)^2.
\end{equation}

 Therefore, in order to prove Blum's conjecture, it suffices to prove that
\begin{equation}  \label{refinement4}
\M(G_1)=13^{a^2}14^{\lfloor\frac{a^2}{4}\rfloor}.
\end{equation}
(Note that (\ref{refinement3}) and (\ref{refinement4}) show that the number of tilings of the hexagonal dungeon $H_{a,2a,b}$ is independent of $b$, for $b\geq 2a$.)

The graphical condensation identity (\ref{kuothm}) involves five new graphs besides the original graph $G$. If all these graphs belong to families of graphs whose number of perfect matchings are conjectured to be given by explicit formulas, then (\ref{kuothm}) provides a way of proving these conjectures by induction. The above family of trapezoidal graphs of type $G_1$ is not large enough for this set-up to hold. However, considering the two more general families of graphs described below (both of which extend the above family of trapezoidal graphs of type $G_1$) will lead to a generalization of Blum's conjecture for which this proof by induction approach will work.


\begin{figure}
\includegraphics[width=10cm]{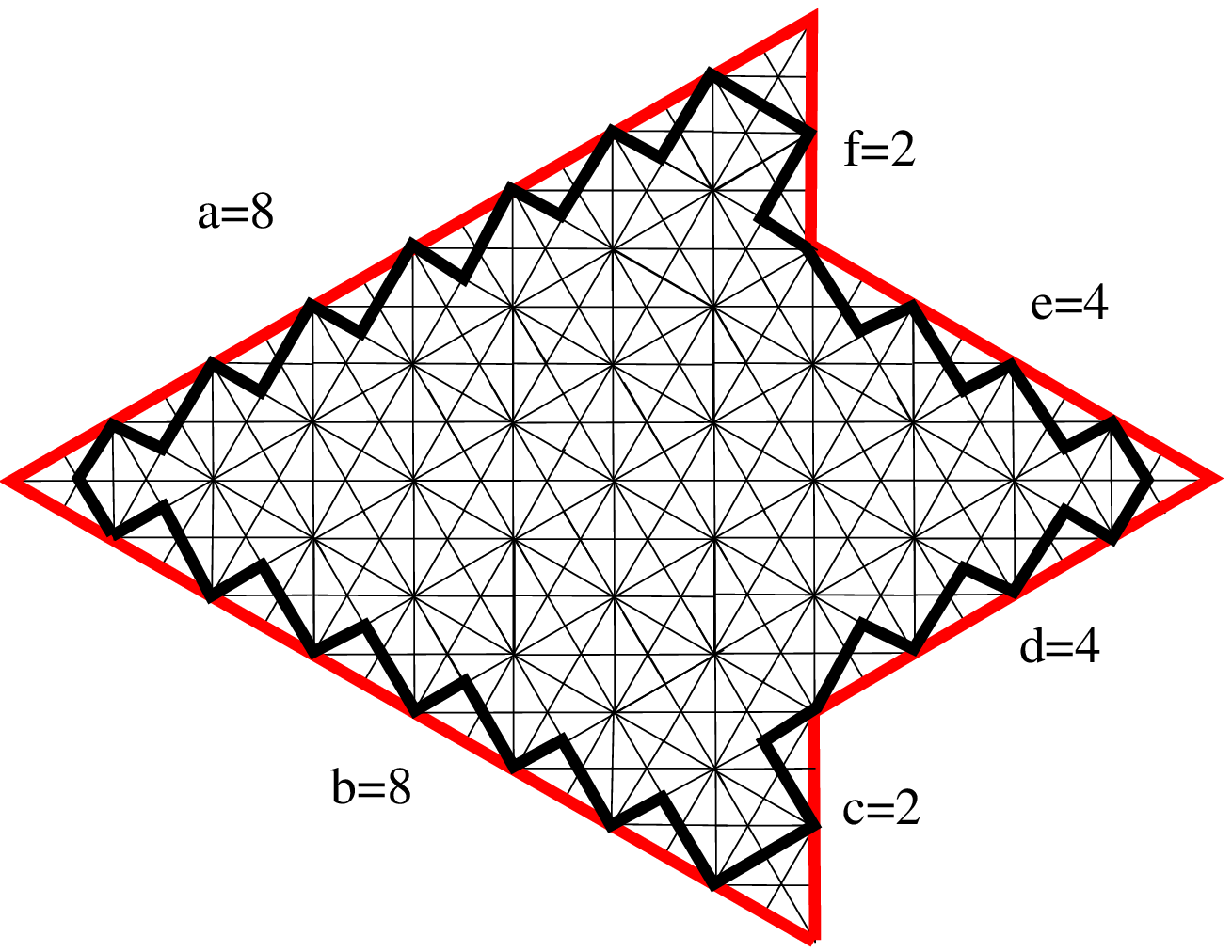}
\caption{The region $D_{8,8,2}$ (an instance of the case $a>c+d$).}
\label{D8,8,2}
\end{figure}

\begin{figure}
\includegraphics[width=10cm]{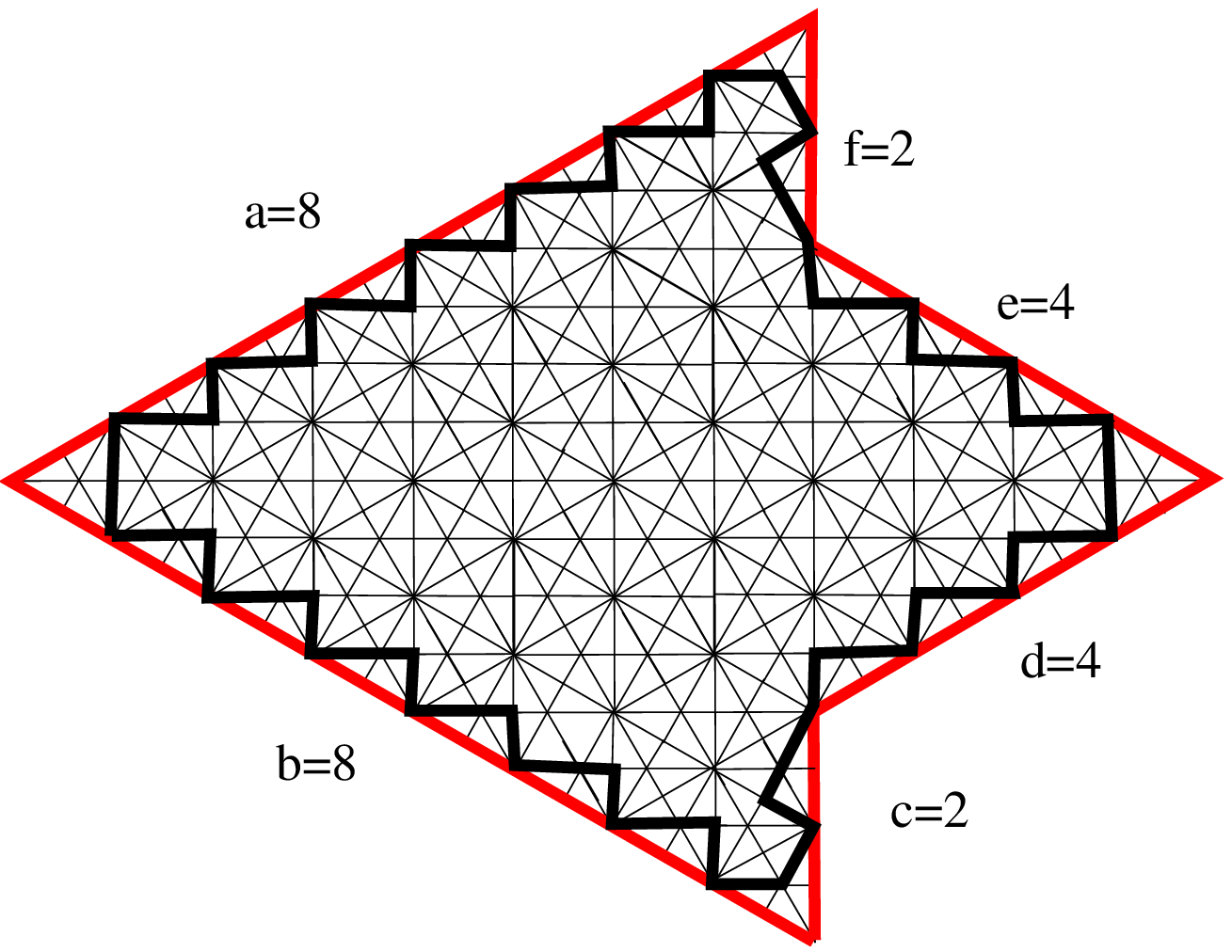}
\caption{The region $E_{8,8,2}$ (an instance of the case $a>c+d$).}
\label{E8,8,2}
\end{figure}

\begin{figure}
\includegraphics[width=8cm]{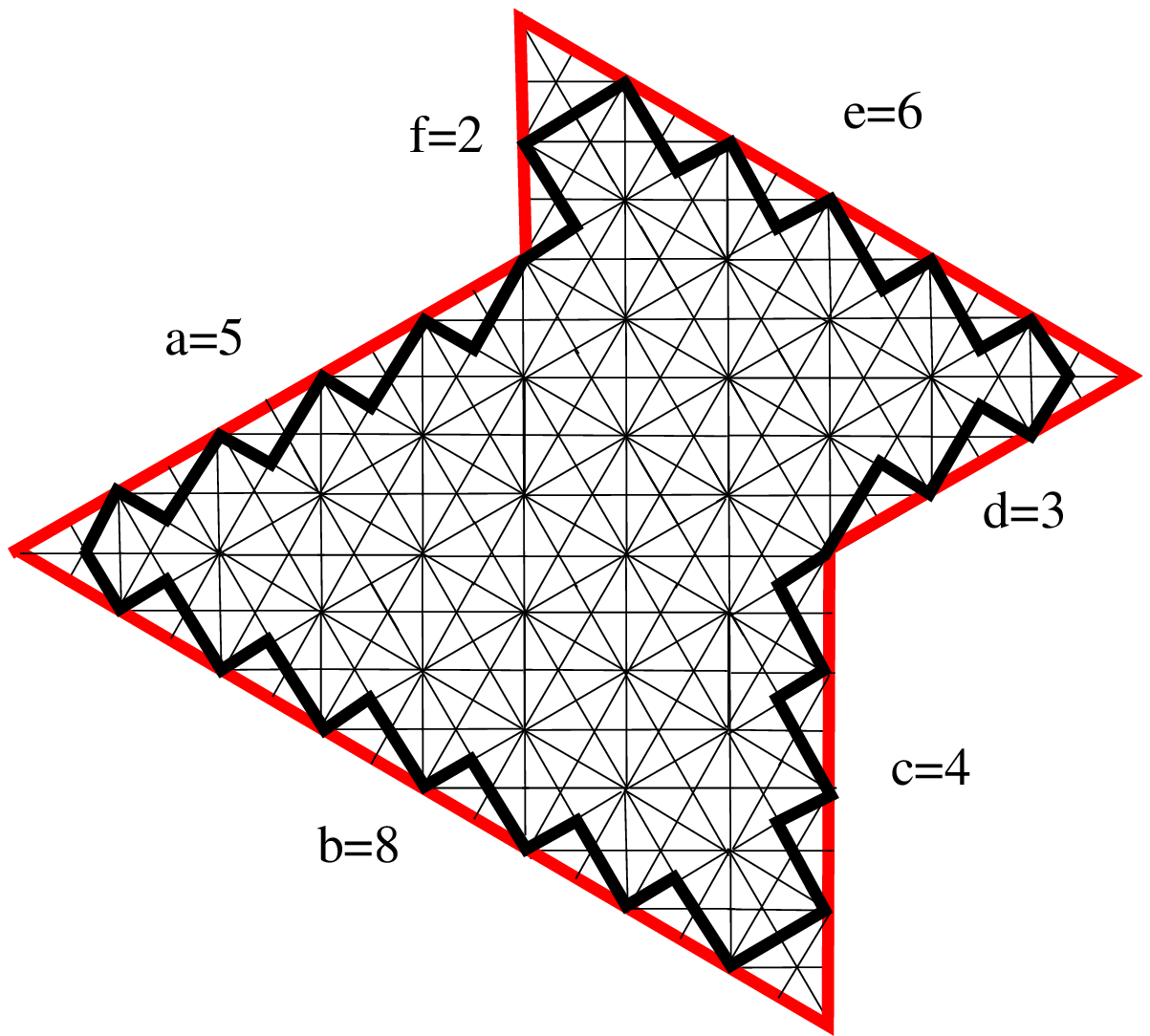}
\caption{The region $D_{5,8,4}$ (an instance of the case $a\leq c+d$).}
\label{D5,8,4}
\end{figure}

\begin{figure}
\includegraphics[width=8cm]{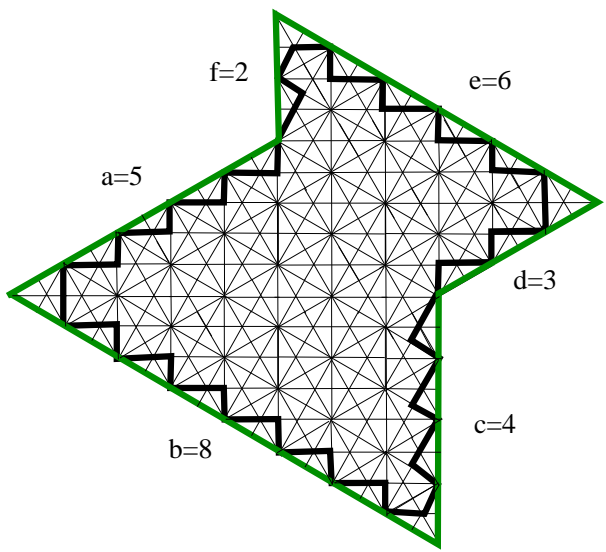}
\caption{The region $E_{5,8,4}$ (an instance of the case $a\leq c+d$).}
\label{E5,8,4}
\end{figure}


For nonnegative integers $a$, $b$ and $c$ (on which some additional constraints will be imposed as we describe our construction), we define a six-sided lattice contour $\mathcal{C}(a,b,c)$ as follows.

Starting from some lattice point, travel along lattice lines $a$ units southwest (a unit being the side-length of the unit triangles), $b$ units southeast, $c$ units north, $d$ units northeast, and $e$ units northwest. Choose $e$ so that the ending point is on the same vertical lattice line as the starting point. One readily sees that this amounts to
\begin{equation}\label{add1}
a+e=b+d.
\end{equation}
Then two different situations may occur: On the common vertical, the ending point is either strictly below the starting point (such an instance is illustrated in Figure \ref{D8,8,2}), or at least as high as the starting point (see Figure \ref{D5,8,4}).

In the first situation, close the contour by traveling $f$ units north. This leads to a closed contour precisely if
\begin{equation}\label{add2}
a+b=2c+d+e+2f.
\end{equation}
Since by (\ref{add1}) $e=b+d-a$,  (\ref{add2}) shows that in this situation we have
\begin{equation}\label{add3}
f=a-c-d.
\end{equation}

In the second situation, close the contour by traveling $f$ units south. This leads to a closed contour precisely if
\begin{equation}\label{add2p}
a+b+2f=2c+d+e.
\end{equation}
Using (\ref{add1}) and (\ref{add2p}), we see that in this situation
\begin{equation}\label{add3p}
f=c+d-a.
\end{equation}

Therefore, in both situations we have
\begin{equation}\label{add4}
e=b+d-a
\end{equation}
and
\begin{equation}\label{add5}
f=|c+d-a|.
\end{equation}
In addition, we will see shortly that we may assume without loss of generality that
\begin{equation}\label{add6}
d=2b-a-2c.
\end{equation}

Denote therefore, for nonnegative integers $a$, $b$ and $c$, by $\mathcal{C}(a,b,c)$ the six-sided lattice contour described above, where $d=2b-a-2c\geq0$, $e=b+d-a=3b-2a-2c\geq0$, and $f=|c+d-a|=|2a-2b+c|$.

Based on the contour $\mathcal{C}(a,b,c)$, we define two lattice regions $D_{a,b,c}$ and $E_{a,b,c}$  determined by the bold jagged contours as in Figures \ref{D8,8,2} and \ref{E8,8,2}, respectively, for the case $a>c+d$, and by Figures \ref{D5,8,4} and \ref{E5,8,4}, respectively, for the case $a\leq c+d$. We note that in order for the thus defined regions to be non-empty, we need to have $b\geq2$. These two families of regions provide the extension we need in order to prove our results by graphical condensation.

Note that the original graph $G_1$ of equation (\ref{refinement4}) is isomorphic to the dual graphs of both $D_{2a,3a,2a}$ and $E_{2a,3a,2a}$  (by (\ref{add5}) and (\ref{add6}), in the corresponding contour we have $d=0$ and $f=0$, and the contour becomes a trapezoid as in Figure \ref{trapezoid}).

\begin{figure}
\includegraphics[width=10cm]{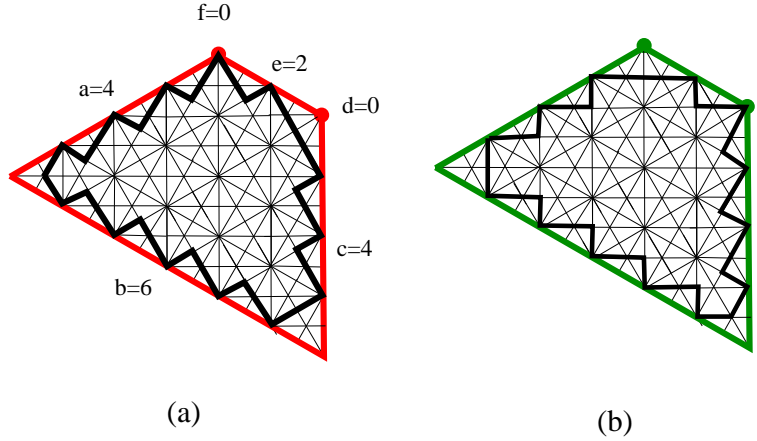}
\caption{The regions $D_{4,6,4}$ (a) and $E_{4,6,4}$ (b).}
\label{trapezoid}
\end{figure}

It will be useful to collect together all the constraints on the non-negative integers $a$, $b$ and $c$ for which the regions $D_{a,b,c}$ and $E_{a,b,c}$ are defined:
\begin{equation}\label{add7}
a\geq0
\end{equation}
\begin{equation}\label{add8}
b\geq2
\end{equation}
\begin{equation}\label{add9}
c\geq0
\end{equation}
\begin{equation}\label{add10}
2b-a-2c\geq0
\end{equation}
\begin{equation}\label{add11}
3b-2a-2c\geq0
\end{equation}
(the last two hold by (\ref{add6}) and (\ref{add4}), respectively).

Since we are interested in the number of tilings of these regions, and as the dual graphs of these regions are bipartite, there is an additional constraint on the values of their side-lengths coming from the condition that the number of vertices in the two bipartition classes are the same. It is easy to check that this amounts to
\begin{equation}\label{add12}
a+c+d=b+e+f
\end{equation}
when $a>c+d$, and
\begin{equation}\label{add12'}
a+c+d+f=b+e
\end{equation}
when $a\leq c+d$.
In both cases, equation (\ref{add6}) follows now by (\ref{add4}) and (\ref{add5}).

We now describe the expressions which give the number of tilings of the above defined regions  $D_{a,b,c}$ and $E_{a,b,c}$.

Let $a,b,c$ be three integers. Define two functions $\phi$ and $\psi$ by setting

\begin{equation}\label{fipsistart}
\phi(a,b,c):=h(a,b,c)13^{g(a,b,c)}14^{f(a,b,c)}
\end{equation}
and
\begin{equation}
\psi(a,b,c):=p(a,b,c) 13^{g(a,b,c)}14^{f(a,b,c)},
\end{equation}
 where
\begin{equation}
g(a,b,c)=(b-a)(b-c)+\left\lfloor\frac{(a-c)^2}{3}\right\rfloor,
\end{equation}
\begin{equation}
f(a,b,c)=  \left\lfloor\frac{(a-b+c)^2}{4}\right\rfloor,
\end{equation}
\begin{equation}
h(a,b,c)=
\begin{cases}
2 &\text{if $3b+a-c\equiv 4\pmod{6}$}\\
3 &\text{if $3b+a-c\equiv 1\pmod{6}$}\\
5 &\text{if $3b+a-c\equiv 5\pmod{6}$}\\
1 &\text{otherwise,}
\end{cases}
\end{equation}
 and
 \begin{equation}\label{fipsiend}
p(a,b,c)=
\begin{cases}
2 &\text{if $3b+a-c\equiv 2\pmod{6}$}\\
3 &\text{if $3b+a-c\equiv 5\pmod{6}$}\\
5 &\text{if $3b+a-c\equiv 1\pmod{6}$}\\
1 &\text{otherwise.}
\end{cases}
\end{equation}

\begin{theorem} \label{main}
Assume that $a$, $b$, and $c$ are three nonnegative integers satisfying $b\geq 2$, $2b-a-2c\geq0$ and $3b-2a-2c\geq0$. Then
\begin{equation}\label{formulas}
\M(D_{a,b,c})=\phi(a,b,c) \text{ and } \M(E_{a,b,c})=\psi(a,b,c)
\end{equation}
\end{theorem}

A special case of this theorem gives a proof for Blum's conjecture on the number of tilings of hexagonal dungeons (Conjecture \ref{MBconjecture}).

\begin{corollary} \label{MB}
Assume that $a$ and $b$ are two positive integers so that $b\geq 2a$.  Then the number of tilings of the hexagonal dungeon
$HD_{a,2a,b}$
is $13^{2a^2}14^{\lfloor\frac{a^2}{2}\rfloor}$.
\end{corollary}

\begin{proof}
As shown at the beginning of this section, it suffices to prove equation (\ref{refinement4}). However, as we noted above, the graph $G_1$ is isomorphic to the dual graph of the region $D_{2a,3a,2a}$. One readily checks that formula (\ref{formulas}) specializes in this case to~(\ref{refinement4}).
\end{proof}

Another special case of Theorem \ref{main} gives the Aztec dungeon theorem (Theorem 3.10 in \cite{Ciucu}). Indeed, one readily sees that the region $E_{n+1,n+1,0}$ is precisely the Aztec dungeon of order $n$ (see Figure \ref{dungeon} for an illustration). Note that this constitutes a new proof of the Aztec dungeon theorem.

Viewed this way, Theorem \ref{main} is seen as a common generalization of the Aztec dungeon theorem and Blum's conjecture on the hexagonal dungeons.

\begin{figure}\centering
\includegraphics[width=10cm]{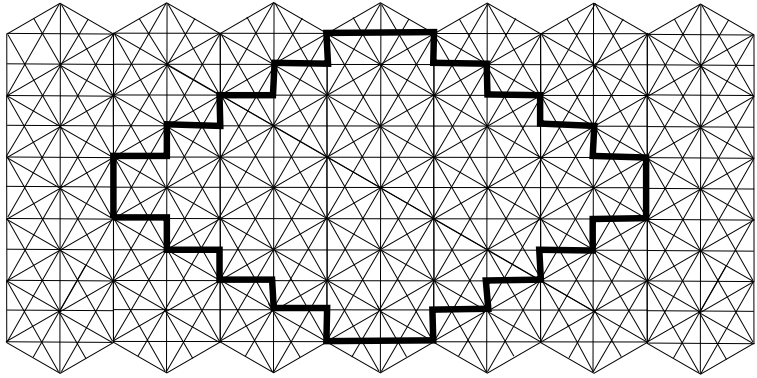}
\caption{The Aztec dungeon of order 5}
\label{dungeon}
\end{figure}



\section{ Recurrences for $\M(D_{a,b,c})$ and $\M(E_{a,b,c})$}

We use Kuo's graphical condensation method (as stated in Theorem \ref{kuothm}) to obtain five recurrences for the number of tilings of the regions $D_{a,b,c}$, and five similar recurrences for the number of tilings of the regions $E_{a,b,c}$. These recurrences are presented in the following three lemmas.

\begin{lemma}\label{con1}  Let $a$, $b$ and $c$ be nonnegative integers so that  $b\geq5$ and $c\geq 2$. Let $d=2b-a-2c$, and assume in addition that $a\geq c+d+1$. Then
\begin{equation} \label{eq1-1}
\begin{split}
\M(D_{a,b,c})\M(D_{a-3,b-3,c-2})=\M(D_{a-2,b-1,c})\M(D_{a-1,b-2,c-2})\\+\M(D_{a-1,b-1,c-1})\M(D_{a-2,b-2,c-1}),
\end{split}\end{equation}
 and
\begin{equation}\label{eq2-1}
 \begin{split}
\M(E_{a,b,c})\M(E_{a-3,b-3,c-2})=\M(E_{a-2,b-1,c})\M(E_{a-1,b-2,c-2})\\+\M(E_{a-1,b-1,c-1})\M(E_{a-2,b-2,c-1}).
\end{split}
\end{equation}
\end{lemma}

\begin{figure}
\includegraphics[width=10cm]{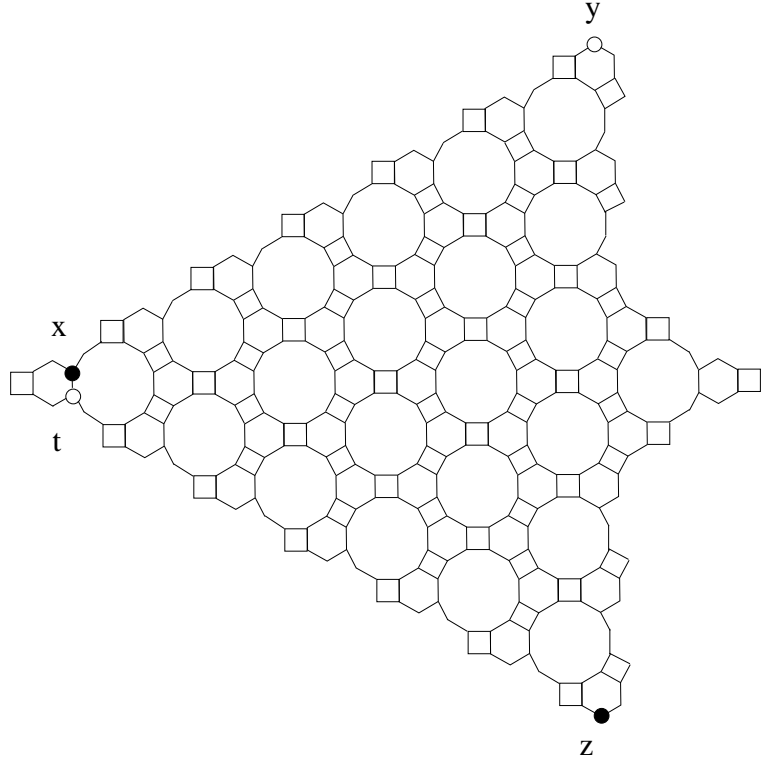}
\caption{The dual graph of the region $D_{8,8,3}$ and the four vertices $x,y,z,t$.}
\label{dualD883}
\end{figure}

\begin{figure}
\includegraphics[width=10cm]{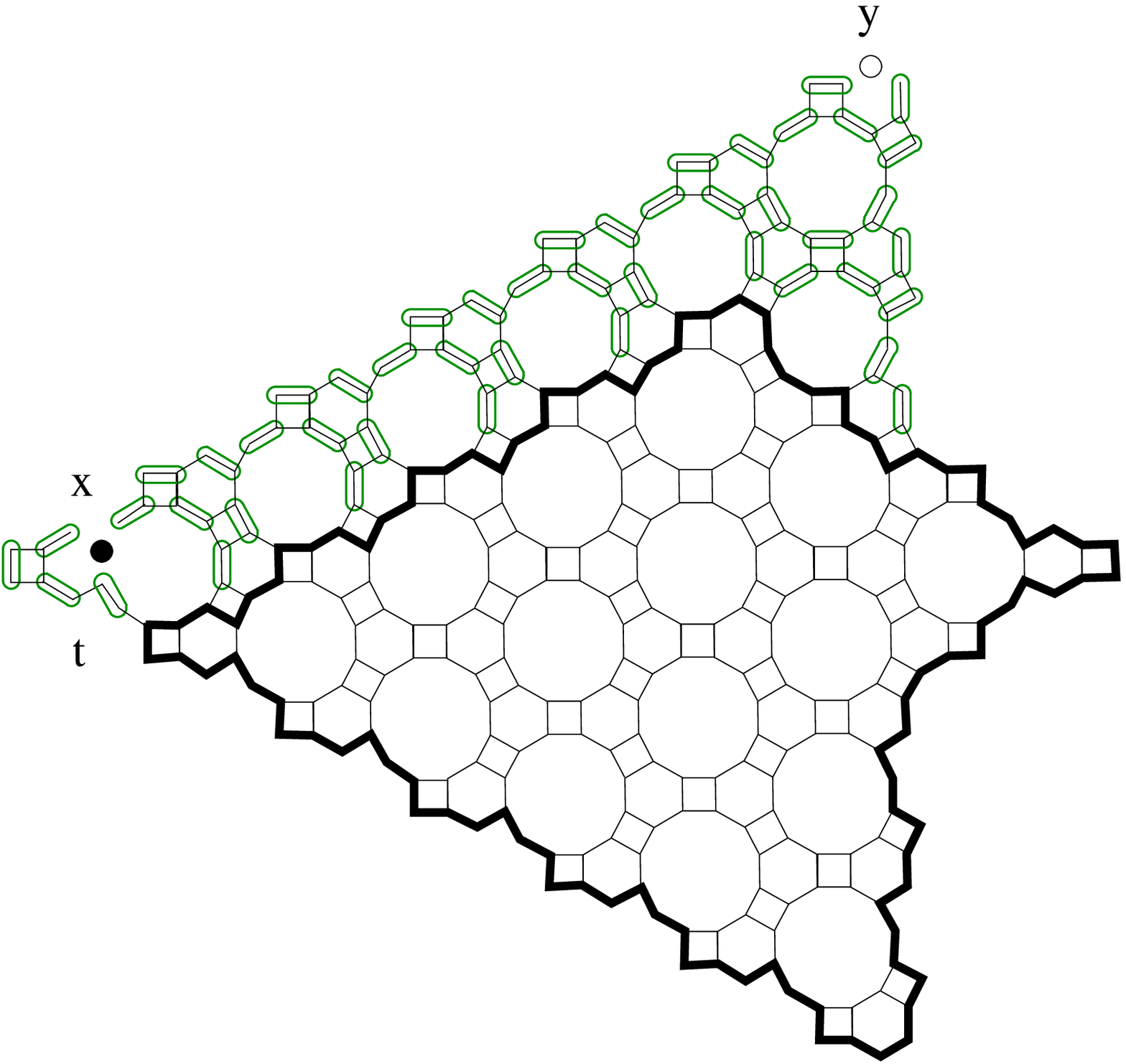}
\caption{Obtaining the dual graph of $D_{6,7,3}$ from the dual graph of $D_{8,8,3}$ by removing vertices $x$ and $y$.}
\label{dualD883xy}
\end{figure}

\begin{figure}
\includegraphics[width=10cm]{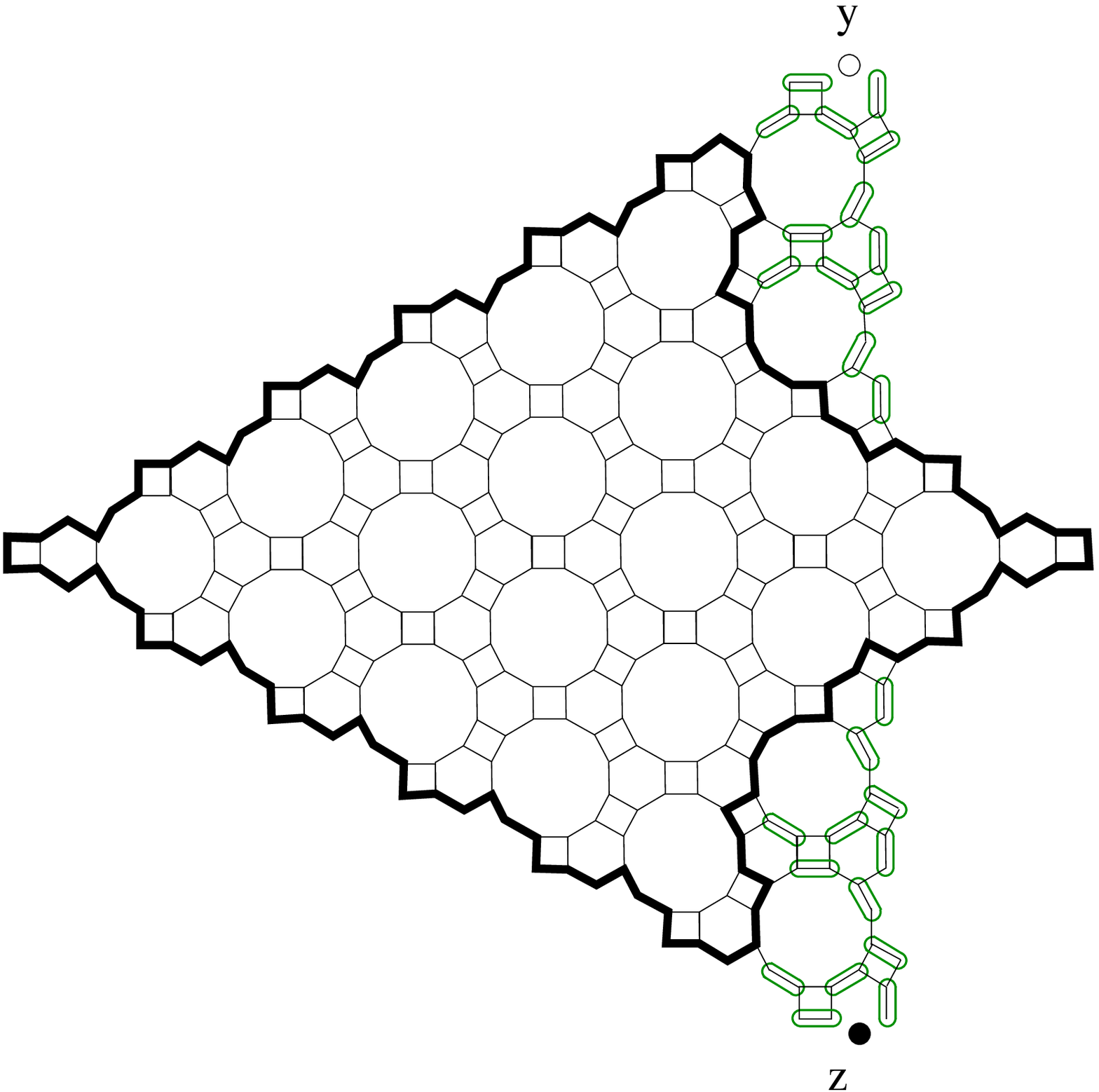}
\caption{Obtaining the dual graph of $D_{7,7,2}$ from the dual graph of $D_{8,8,3}$ by removing vertices $y$ and $z$.}
\label{dualD883yz}
\end{figure}

\begin{figure}
\includegraphics[width=10cm]{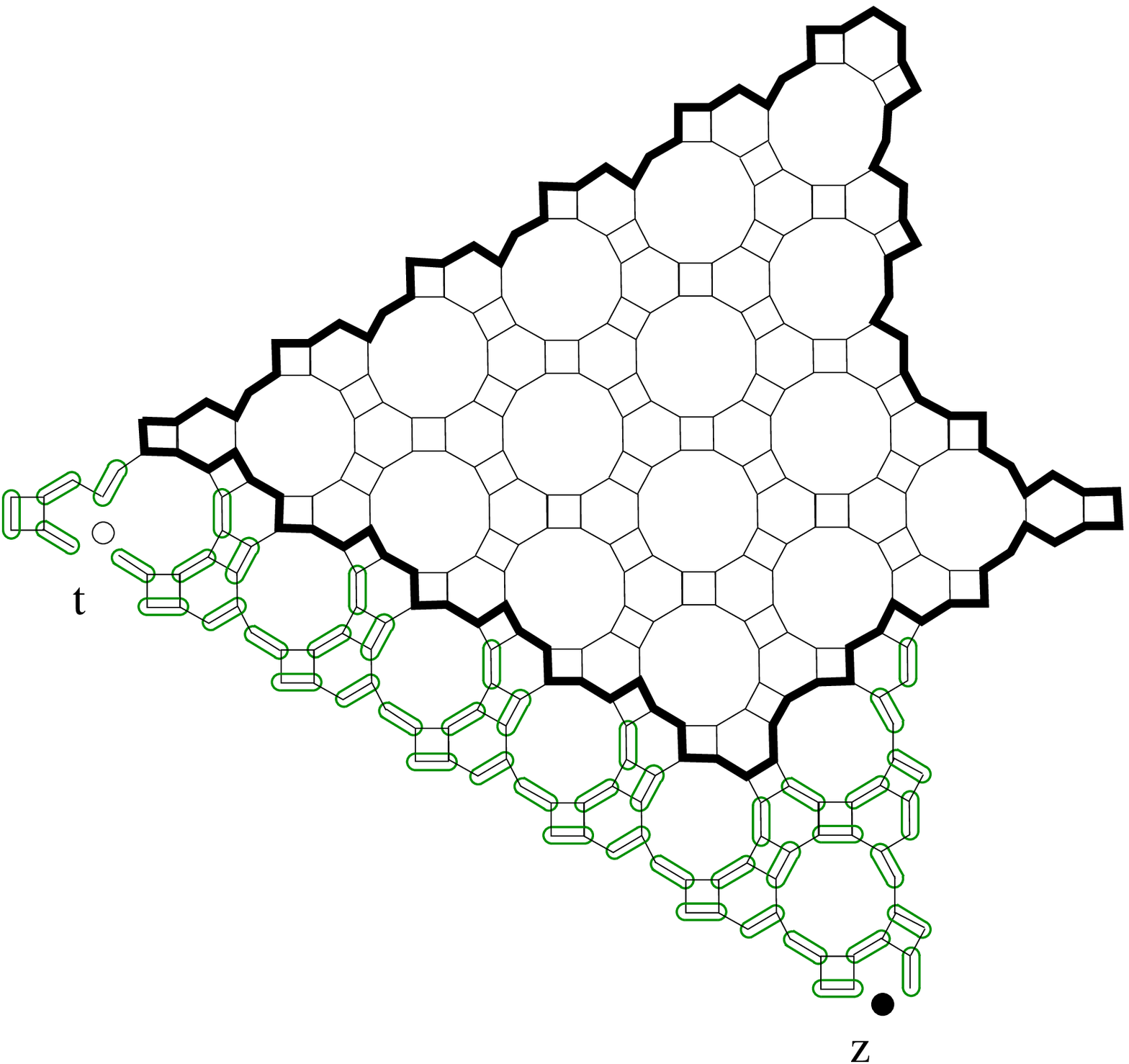}
\caption{Obtaining the dual graph of $D_{7,6,1}$ from the dual graph of $D_{8,8,3}$ by removing vertices $z$ and $t$.}
\label{dualD883zt}
\end{figure}

\begin{figure}
\includegraphics[width=10cm]{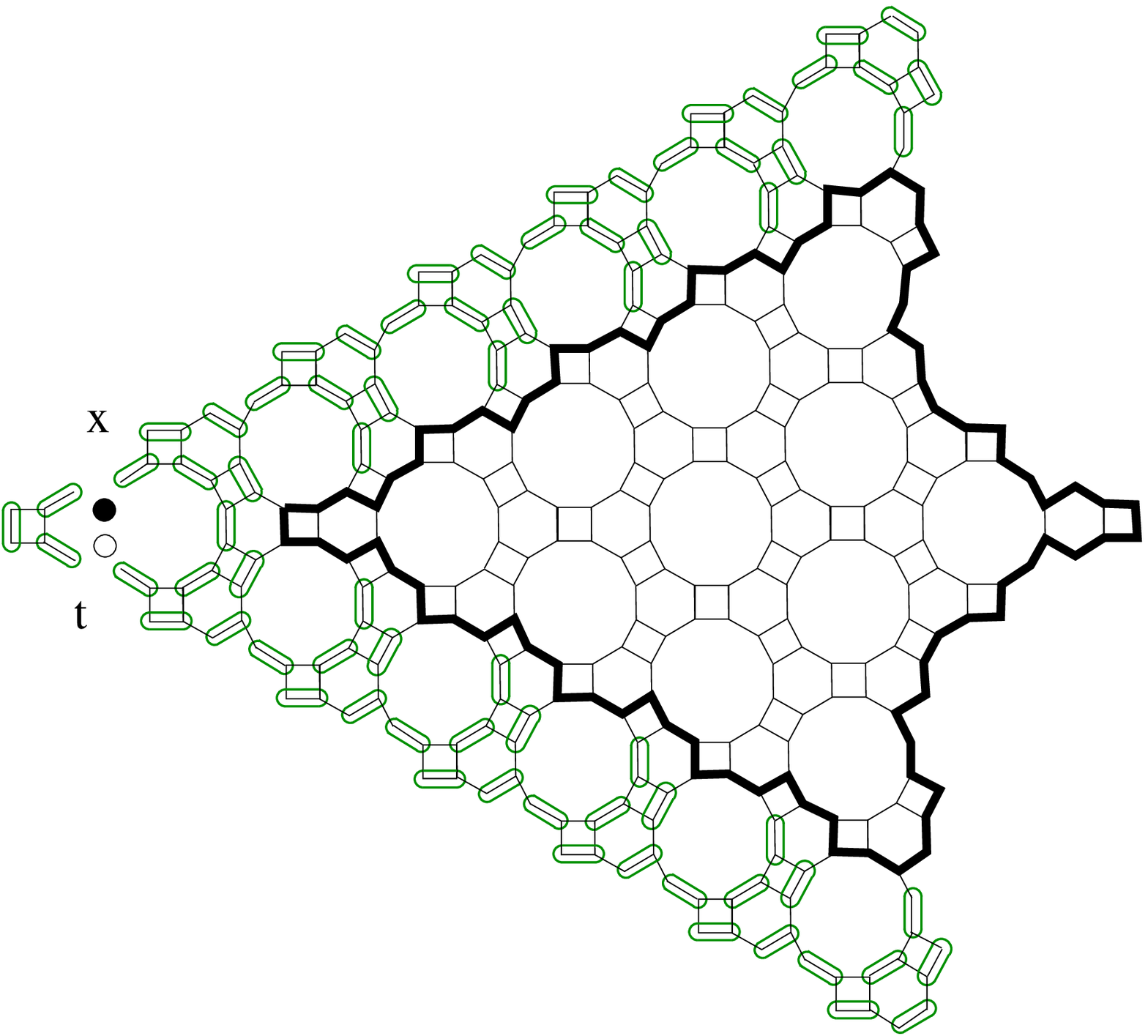}
\caption{Obtaining the dual graph of $D_{6,6,2}$ from the dual graph of $D_{8,8,3}$  by removing vertices $t$ and $x$.}
\label{dualD883tx}
\end{figure}

\begin{figure}
\includegraphics[width=10cm]{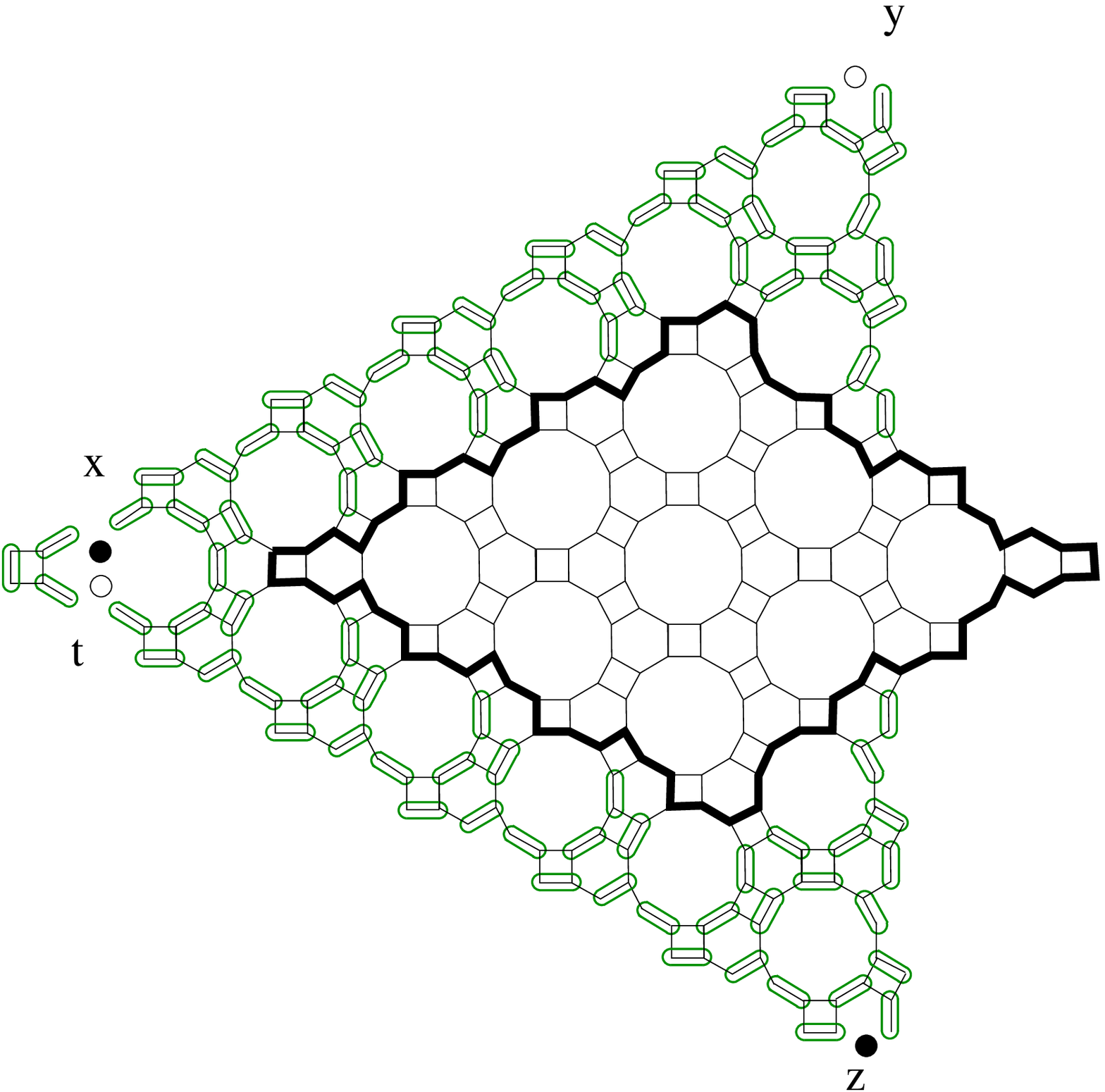}
\caption{Obtaining the dual graph of $D_{5,5,1}$ from the dual graph of $D_{8,8,3}$ by removing vertices $x,y,z$ and $t$.}
\label{dualD883xyzt}
\end{figure}

\begin{proof}
One readily checks that if $a$, $b$ and $c$ satisfy conditions (\ref{add7})--(\ref{add11}), then so do all the other five triples of indices that occur in (\ref{eq1-1}) and (\ref{eq2-1}). Therefore, if the region $D_{a,b,c}$ is defined, so are the other five $D$-type regions in (\ref{eq1-1}). A similar statement holds for equation (\ref{eq2-1}).

We prove first recurrence (\ref{eq1-1}). Denote by $G$ the dual graph of $D_{a,b,c}$. Consider four vertices $x,y,z,t$ as indicated in Figure \ref{dualD883} (which corresponds to $a=8$, $b=8$ and $c=3$). Namely, $x$ and $t$ are near the western ``corner" of the graph $G$, $y$ is at the northern corner, and $z$ is at the southern corner.  By removing the forced edges from the graph $G-\{x,y\}$, we obtain a graph isomorphic to the dual $H$ of the region $D_{a-2,b-2,c-1}$. This process is illustrated in Figure \ref{dualD883xy} for the case $a=8,$  $b=8,$ $c=3$; the circled edges are the forced edges, and the boundary of the dual graph of $D_{a-2,b-1,c}$ is indicated by the bold contour. Thus,
\begin{equation}\label{con1a1}
\M(G-\{x,y\})=\M(H)=\M(D_{a-2,b-1,c}).
\end{equation}
  Similarly, we get

\begin{equation}\label{con1a2}
\M(G-\{y,z\})=\M(D_{a-1,b-1,c-1}) \text{ (see Figure \ref{dualD883yz}),}
\end{equation}

\begin{equation}\label{con1a3}
\M(G-\{z,t\})=\M(D_{a-1,b-2,c-2}) \text{ (see Figure \ref{dualD883zt}),}
\end{equation}

 \begin{equation}\label{con1a4}
 \M(G-\{t,x\})=\M(D_{a-2,b-2,c-1}) \text{ (see Figure \ref{dualD883tx}),}
 \end{equation}

and
\begin{equation}\label{con1a5}
\M(G-\{x,y,z,t\})=\M(D_{a-3,b-3,c-2}) \text{ (see Figure \ref{dualD883xyzt}).}
\end{equation}
 Therefore, Theorem \ref{kuothm} and the five equalities (\ref{con1a1})--(\ref{con1a5}) imply (\ref{eq1-1}).

\begin{figure}
\includegraphics[width=8cm]{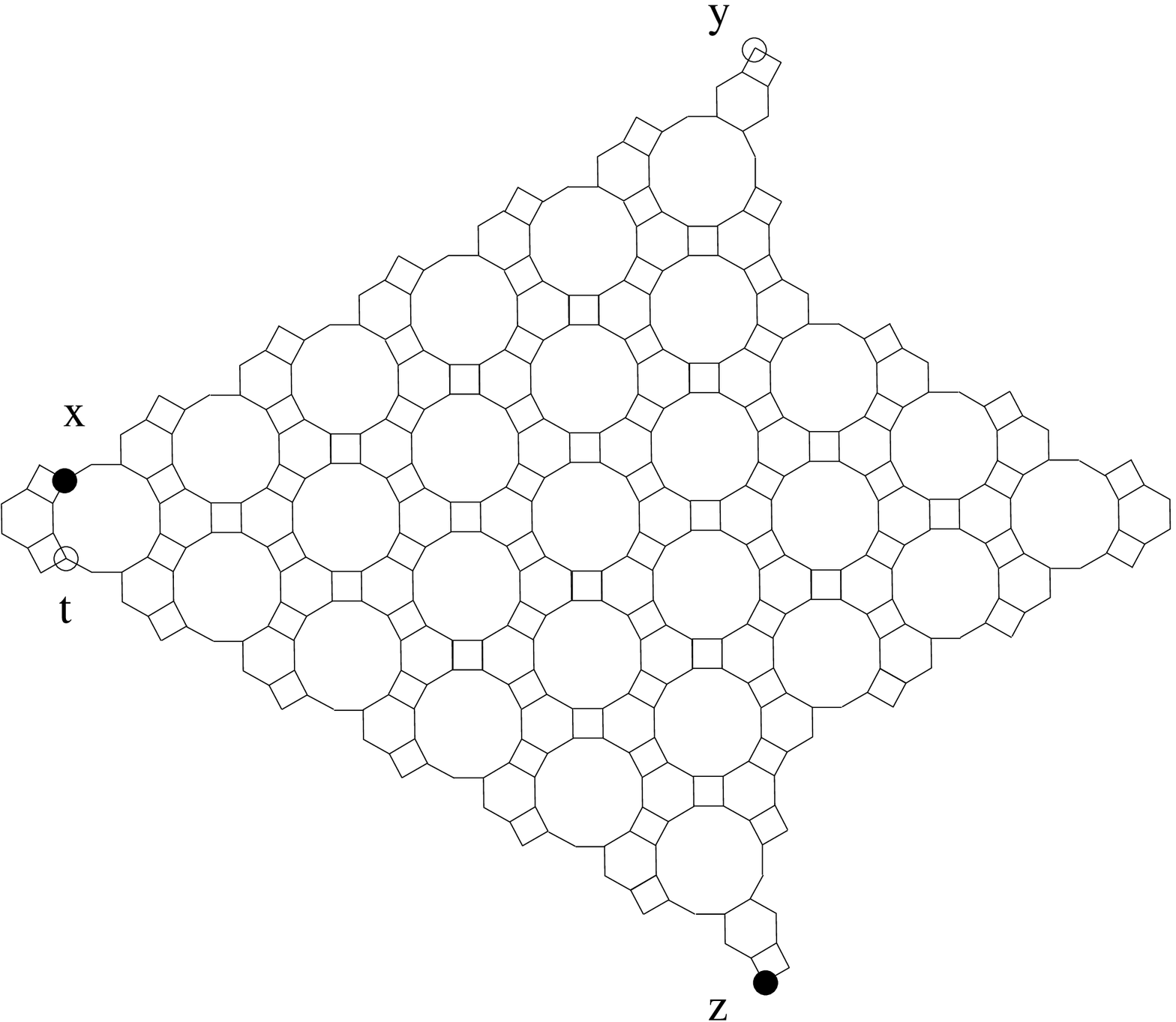}
\caption{The dual graph of $E_{8,8,2}$ and the four vertices $x,y,z,t$.}
\label{conden2}
\end{figure}
Recurrence (\ref{eq2-1}) can be proved similarly, by applying Theorem \ref{kuothm} to the dual graph of $E_{a,b,c}$, with vertices $x,y,z,t$ chosen at its corners as in shown in Figure \ref{conden2} (which illustrates the case $a=8,\ b=8,\ c=2$).
\end{proof}

\begin{lemma}\label{con3}    Let $a$, $b$ and $c$ be nonnegative integers satisfying $a\geq 2$, $b\geq4$, $2b-a-2c\geq2$, and $3b-2a-2c\geq2$.

$(${\rm a}$)$. If $c\geq 1$, then
  \begin{equation}\label{eq3-3}
  \begin{split}
   \M(D_{a,b,c})\M(D_{a-2,b-2,c})=\M(D_{a-1,b-1,c})^2\\+\M(D_{a,b,c+1})\M(D_{a-2,b-2,c-1})
  \end{split}
  \end{equation}
  and
  \begin{equation}\label{eq4-3}
  \begin{split}
   \M(E_{a,b,c})\M(E_{a-2,b-2,c})=\M(E_{a-1,b-1,c})^2\\+\M(E_{a,b,c+1})\M(E_{a-2,b-2,c-1}).
  \end{split}
  \end{equation}

$(${\rm b}$)$. If  $c=0$, then
   \begin{equation}\label{eq3-4}
     \begin{split}
      \M(D_{a,b,0})\M(D_{a-2,b-2,0})=\M(D_{a-1,b-1,0})^2\\+\M(D_{a,b,1})\M(D_{e,d,1})
      \end{split}
   \end{equation}
and
   \begin{equation}\label{eq4-4}
     \begin{split}
      \M(E_{a,b,0})\M(E_{a-2,b-2,0})=\M(E_{a-1,b-1,0})^2\\+\M(E_{a,b,1})\M(E_{e,d,1}),
      \end{split}
   \end{equation}
where, as usual, $d=2a-b-2c$ and $e=3a-2b-2c$.
    \end{lemma}

\begin{proof}

\begin{figure}\centering
\includegraphics[width=8cm]{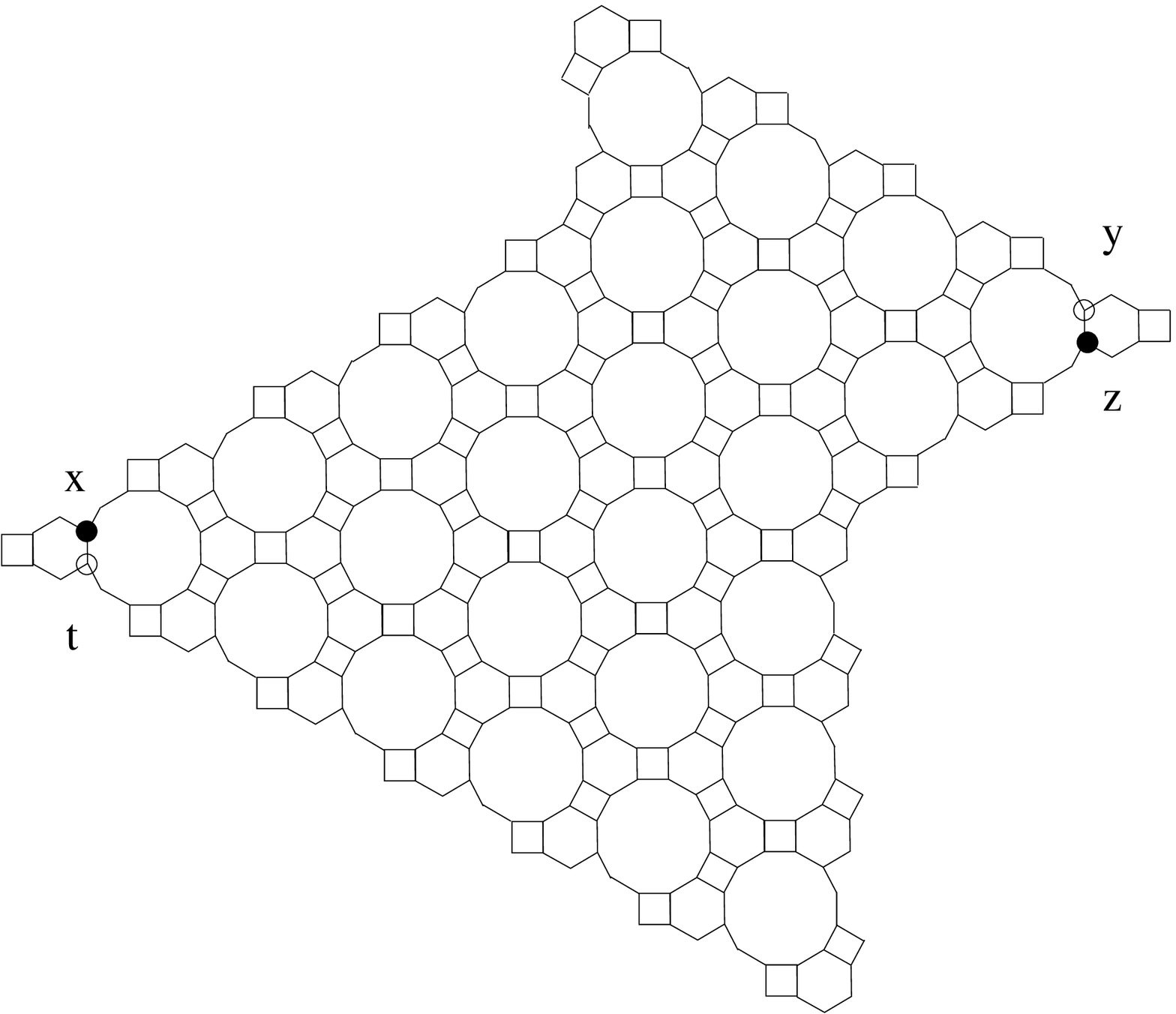}
\caption{The dual graph of $D_{5,8,4}$ and the four vertices $x,y,z,t$.}
\label{condensation4n}
\end{figure}

\begin{figure}\centering
\includegraphics[width=10cm]{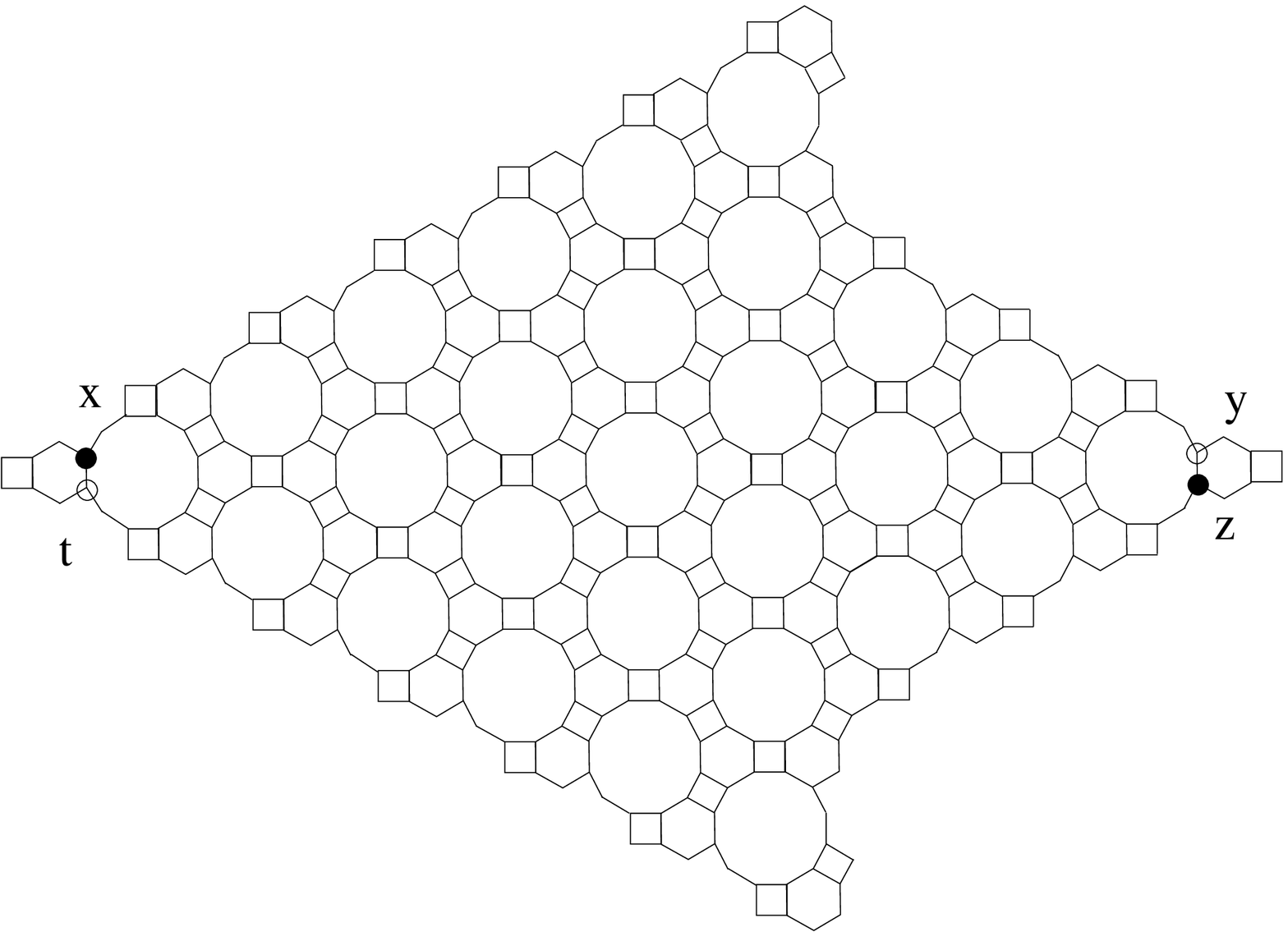}
\caption{The dual graph of $D_{8,8,2}$  and  the four vertices $x,y,z,t$.}
\label{condensation1n}
\end{figure}

As in the proof of Lemma \ref{con1}, one can readily check that if the region $D_{a,b,c}$ is defined (i.e., if $a$, $b$ and $c$ satisfy conditions (\ref{add7})--(\ref{add11})), then all the other five regions in (\ref{eq3-3}) are also defined. Analogous statements hold for the recurrences (\ref{eq4-3})--(\ref{eq4-4}).

(a). We present the proof of (\ref{eq3-3}) (the proof of (\ref{eq4-3}) is perfectly analogous).

In the case $a\leq c+d$, apply Theorem \ref{kuothm} to the dual graph $G$ of $D_{a,b,c}$ with the vertices $x,y,z,t$ chosen as shown in Figure \ref{condensation4n} (in that figure, the values of the parameters are $a=5$, $b=8$ and $c=4$). On the other hand, if $a> c+d$, apply Theorem \ref{kuothm} to the dual graph $G$ of $D_{a,b,c}$ with the vertices $x,y,z,t$ chosen as shown in Figure \ref{condensation1n} (where $a=8$, $b=8$ and $c=2$). Arguing similarly as in the proof of Lemma \ref{con1}, we obtain that
\begin{equation}\label{con3a1}
\M(G-\{x,y\})=\M(D_{a-1,b-1,c}),
\end{equation}
\begin{equation}\label{con3a2}
\M(G-\{y,z\})=\M(D_{a,b,c+1}),
\end{equation}
\begin{equation}\label{con3a3}
\M(G-\{z,t\})=\M(D_{a-1,b-1,c}),
\end{equation}
\begin{equation}\label{con3a4}
\M(G-\{t,x\})=\M(D_{a-2,b-2,c-1}),
\end{equation}
\begin{equation}\label{con3a5}
\M(G-\{x,y,z,t\})=\M(D_{a-2,b-2,c}).
\end{equation}
 Therefore, (\ref{eq3-3}) is obtained by substituting the equalities (\ref{con3a1})--(\ref{con3a5}) in the recurrence of  Theorem \ref{kuothm}.

\begin{figure}\centering
\includegraphics[width=12cm]{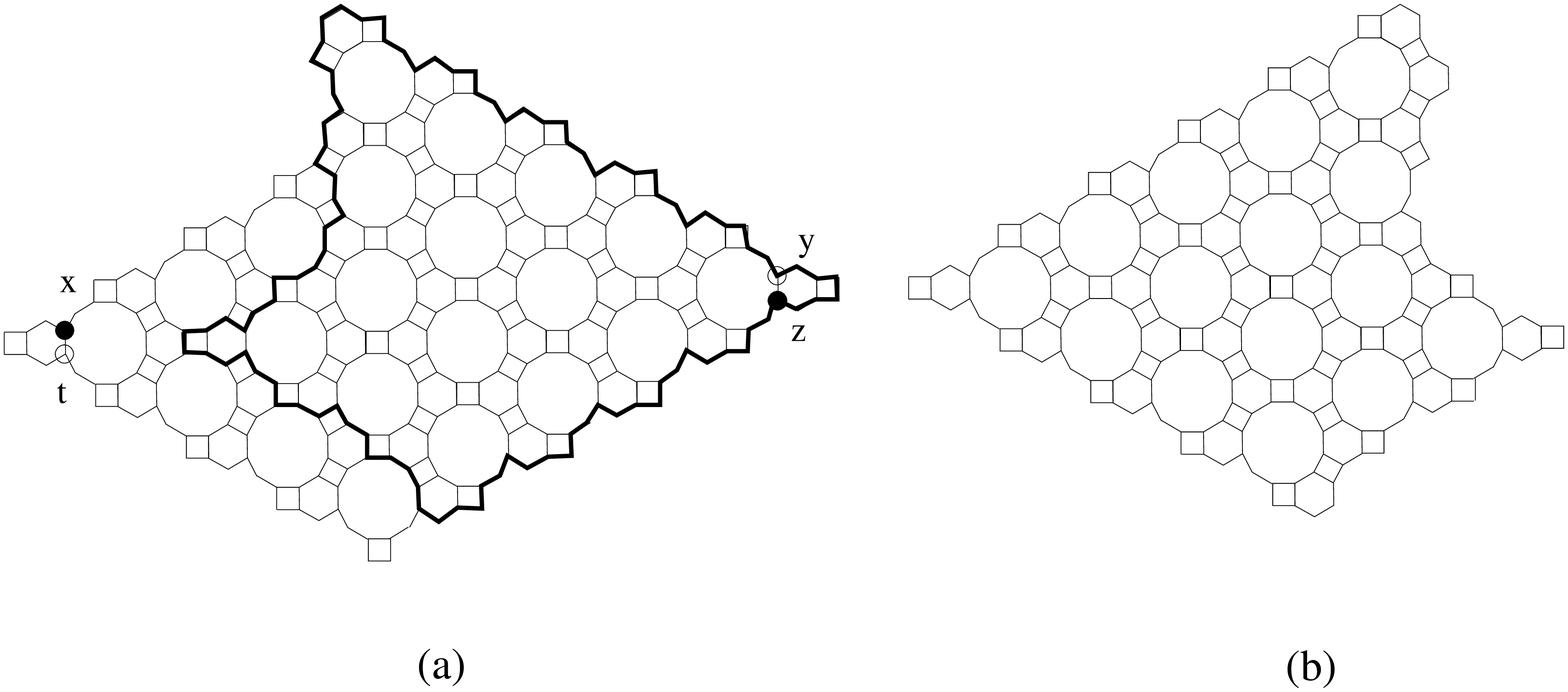}
\caption{The dual graph of $D_{4,5,0}$ (a), and the dual graph  of $D_{7,6,1}$ (b).}
\label{conden9}
\end{figure}

\begin{figure}\centering
\includegraphics[width=12cm]{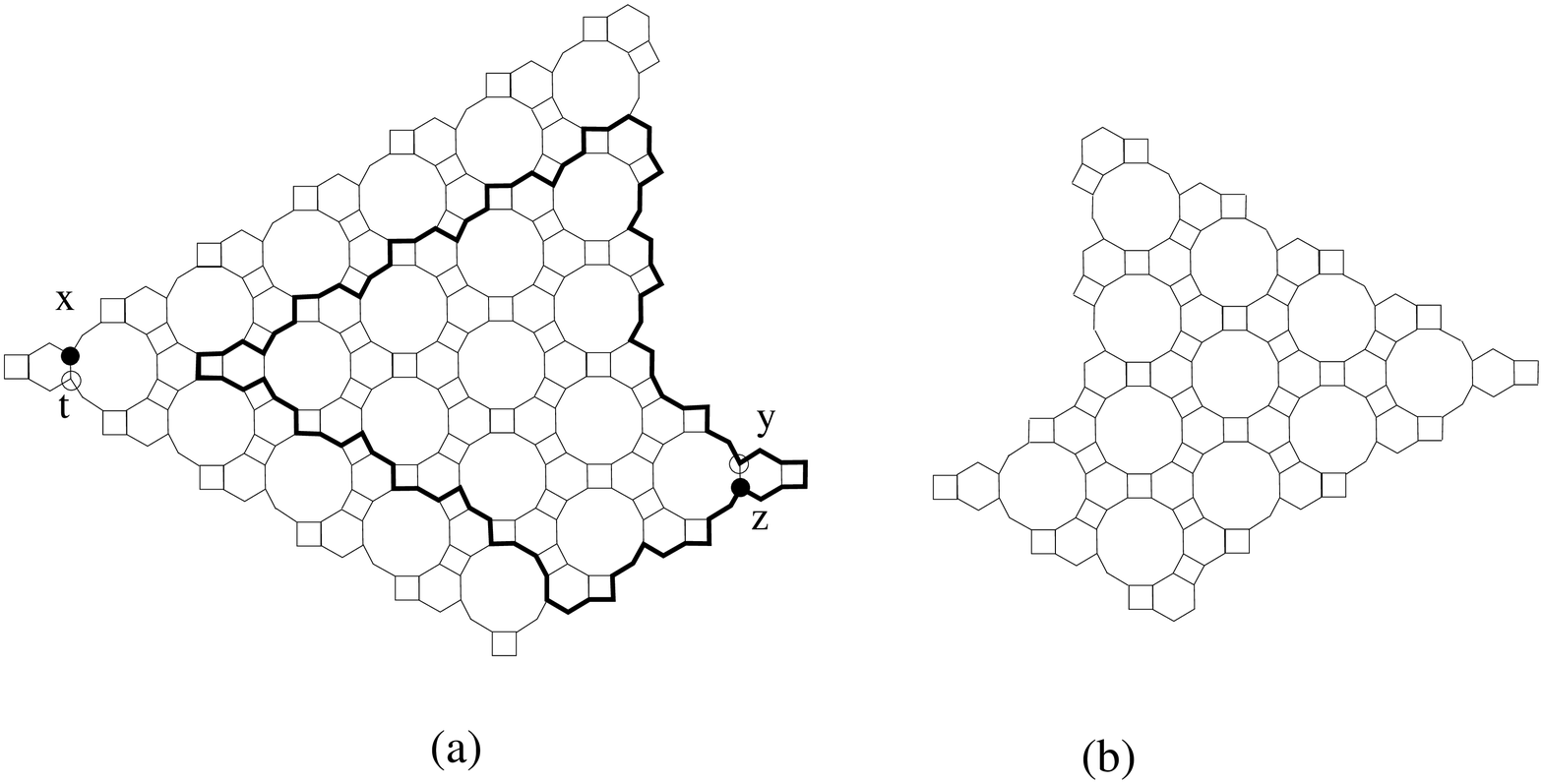}
\caption{The dual graph of $D_{8,6,0}$ (a), and the dual graph of $D_{2,4,1}$ (b).}
\label{condensation3n}
\end{figure}

(b). Again, we prove only (\ref{eq3-4}), as (\ref{eq4-4}) is obtained in a perfectly analogous manner.  Equation (\ref{eq3-4}) can be treated similarly to (\ref{eq3-3}), using Figure \ref{conden9} for the case $a\leq c+d$ ($a=4$, $b=5$ and $c=0$ in Figure \ref{conden9}), and Figure \ref{condensation3n} for the case $a>c+d$ (in Figure \ref{condensation3n}, we have $a=8$, $b=6$ and $c=0$). In other words, we still pick $x,t$ at the western corner, and $y,z$ at the eastern corner of the graph. We still get the four equalities (\ref{con3a1}), (\ref{con3a2}), (\ref{con3a3}), and (\ref{con3a5}). However, the graph obtained from $G-\{x,t\}$ by removing forced edges  (the boundary of this graph is indicated by the bold contours in Figures \ref{conden9}(a) and \ref{condensation3n}(a)) is now isomorphic to the dual graph of $D_{e,d,1}$ (the latter is indicated in Figures \ref{conden9}(b) and \ref{condensation3n}(b)). Thus,
\begin{equation}\label{con3a6}
\M(G-\{x,t\})=\M(D_{e,d,1}),
\end{equation}
and (\ref{eq3-4}) follows Theorem \ref{kuothm}.
\end{proof}


 \begin{lemma}\label{con4} Assume that $a,b,c$ are three nonnegative integers satisfying $a\geq2$, $b\geq5$ and $c\geq2$. As usual, let $d=2b-a-2c$. Assume in addition that $a\leq c+d$.

 $(${\rm a}$)$. If $d\geq1$, then
 \begin{equation}\label{eq3-1}
  \begin{split}
         \M(D_{a,b,c})\M(D_{a-2,b-3,c-2})=\M(D_{a-1,b-1,c})\M(D_{a-1,b-2,c-2})\\+\M(D_{a-2,b-2,c-1})\M(D_{a,b-1,c-1})
  \end{split}
  \end{equation}
  and
 \begin{equation}\begin{split}\label{eq4-1}
        \M(E_{a,b,c})\M(E_{a-2,b-3,c-2})=\M(E_{a-1,b-1,c})\M(E_{a-1,b-2,c-2})\\+\M(E_{a-2,b-2,c-1})\M(E_{a,b-1,c-1}).
  \end{split}
  \end{equation}

 $(${\rm b}$)$. If $d=0$, then
  \begin{equation}\label{eq3-2}
    \begin{split}
       \M(D_{a,b,c})\M(D_{a-2,b-3,c-2})=\M(E_{c,b-1,a-1})\M(D_{a-1,b-2,c-2})\\+\M(D_{a-2,b-2,c-1})\M(D_{a,b-1,c-1})
    \end{split}
    \end{equation}
and
  \begin{equation}\label{eq4-2}
    \begin{split}
    \M(E_{a,b,c})\M(E_{a-2,b-3,c-2})=\M(D_{c,b-1,a-1})\M(E_{a-1,b-2,c-2})\\+\M(E_{a-2,b-2,c-1})\M(E_{a,b-1,c-1}).
    \end{split}
    \end{equation}
\end{lemma}

{\bf Note.} Unlike in the case of Lemma \ref{con3}, it is not immediately apparent from recurrences (\ref{eq3-1}) and (\ref{eq4-1}) that their form does not hold for $d=0$, because the parameter $d$ of the $D$- and $E$-regions is not displayed. The reason why for $d=0$ the recurrences take on the changed form given in (\ref{eq3-2}) and (\ref{eq4-2}) is that the first regions on the right of the former recurrences have their $d$-parameter one unit smaller that the $d$-parameter of the regions $D_{a,b,c}$ and $E_{a,b,c}$.

 \begin{proof}
(a). We prove first the recurrence (\ref{eq3-1}); (\ref{eq4-1}) is obtained by arguing similarly.

\begin{figure}\centering
\includegraphics[width=8cm]{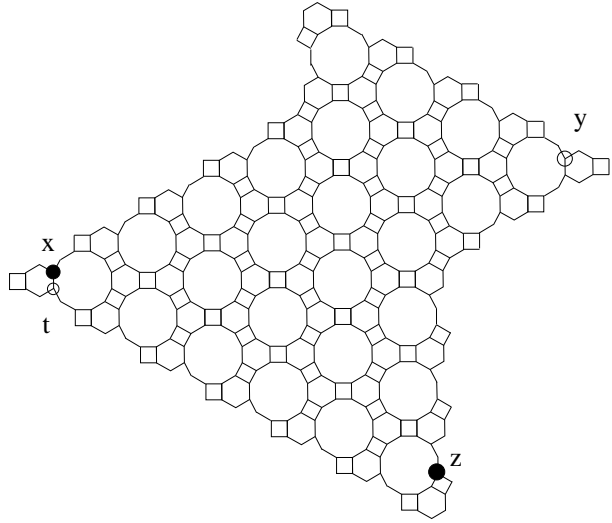}
\caption{The dual graph of $D_{5,8,4}$ and the vertices $x,y,z,t$.}
\label{condensation5n}
\end{figure}

Apply Theorem \ref{kuothm} to the dual graph $G$ of $D_{a,b,c}$ with the vertices $x,y,z,t$ chosen as indicated in Figure \ref{condensation5n}, for $a=5$, $b=8$ and $c=4$. More precisely, we pick $x$ and $t$ at the western corner, $y$ at the eastern corner, and $z$ at the southern corner of $G$. Similarly to the proofs of Theorems \ref{con1} and \ref{con3}, by removing forced edges we get the following facts:
\begin{equation}\label{con4a1}
\M(G-\{x,y\})=\M(D_{a-1,b-1,c}),
\end{equation}
\begin{equation}\label{con4a2}
\M(G-\{y,z\})=\M(D_{a,b-1,c-1}),
\end{equation}
\begin{equation}\label{con4a3}
\M(G-\{z,t\})=\M(D_{a-1,b-2,c-2}),
\end{equation}
\begin{equation}\label{con4a4}
\M(G-\{t,x\})=\M(D_{a-2,b-2,c-1}),
\end{equation}
\begin{equation}\label{con4a5}
\M(G-\{x,y,z,t\})=\M(D_{a-2,b-3,c-2}).
\end{equation}
 Therefore, (\ref{eq3-1}) follows Theorem \ref{kuothm} and five equalities (\ref{con4a1})--(\ref{con4a5}).

(b). Recurrence (\ref{eq3-2}) can be obtained similarly to (\ref{eq3-1}), by choosing the points $x,y,z,t$ as shown (for $a=4$, $b=8$ and $c=6$) in Figure \ref{conden10}. The four equalities (\ref{con4a2})--(\ref{con4a5}) still hold. The only difference is that the graph obtained from $G-\{x,y\}$ by removing forced edges (see the graph restricted by the bold contour in Figure \ref{conden10}(a)) is now isomorphic to the dual graph of $E_{c,b-1, a-1}$ (illustrated in Figure \ref{conden10}(b)). Therefore, we have the following equality instead of (\ref{con4a1})
\begin{equation}
\M(G-\{x,y\})=\M(E_{c,b-1,a-1}),
\end{equation}
and we obtain (\ref{eq3-2}).

\begin{figure}\centering
\includegraphics[width=12cm]{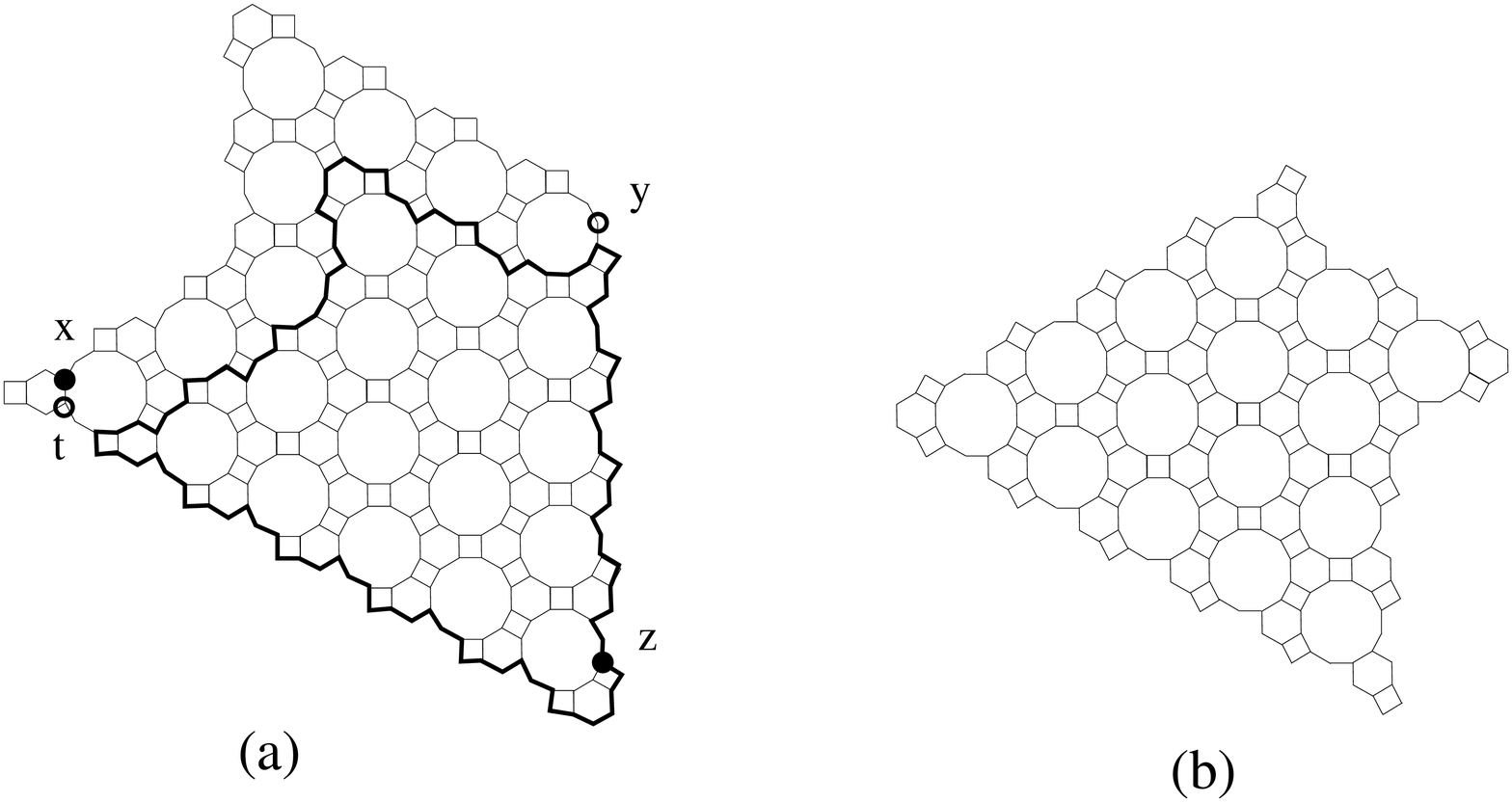}
\caption{The dual graph of $D_{4,8,6}$ (a), and the dual graph of $E_{6,7,3}$ (b). }
\label{conden10}
\end{figure}

The proof of the recurrence (\ref{eq4-2}) is perfectly analogous.
\end{proof}

\section{Recurrences for $\phi(a,b,c)$ and $\psi(a,b,c)$}

We show that the  functions $\phi(a,b,c)$ and $\psi(a,b,c)$ defined by (\ref{fipsistart})--(\ref{fipsiend}) satisfy the same recurrences as the numbers $\M(D_{a,b,c})$ and $\M(E_{a,b,c})$ were shown to satisfy in Section 4.

\begin{lemma}\label{form1} For any integers $a$, $b$ and $c$ we have
\begin{equation}\label{recur1}
\begin{split}
\phi(a,b,c)\phi(a-3,b-3,c-2)=\phi(a-2,b-1,c)\phi(a-1,b-2,c-2)\\+\phi(a-1,b-1,c-1)\phi(a-2,b-2,c-1)
\end{split}
\end{equation}
and
\begin{equation}\label{recur1b}
\begin{split}
\psi(a,b,c)\psi(a-3,b-3,c-2)=\psi(a-2,b-1,c)\psi(a-1,b-2,c-2)\\+\psi(a-1,b-1,c-1)\psi(a-2,b-2,c-1).
\end{split}
\end{equation}
\end{lemma}

\begin{proof}
We verify below (\ref{recur1}). Equation (\ref{recur1b}) is checked in the same fashion.

Using the definitions of the functions $f$ and $g$, one readily verifies the following facts:

 (i) If $a-b+c=2k$, then
\begin{equation}\label{fact1-1}
\begin{split}
f(a,b,c)+f(a-3,b-3,c-2)=f(a-2,b-1,c)+f(a-1,b-2,c-2)+1\\=f(a-1,b-1,c-1)+f(a-2,b-2,c-1)+1=2k^{2}-2k+1
\end{split}
\end{equation}

 (ii) If $a-b+c=2k+1$, then
\begin{equation}\label{fact1-2}
\begin{split}
f(a,b,c)+f(a-3,b-3,c-2)=f(a-2,b-1,c)+f(a-1,b-2,c-2)\\=f(a-1,b-1,c-1)+f(a-2,b-2,c-1)=2k^{2}
\end{split}
\end{equation}

 (iii) If $a-c=3l$ or $3l+1$, then
\begin{equation}\label{fact1-3}
\begin{split}
g(a,b,c)+g(a-3,b-3,c-2)=g(a-2,b-1,c)+g(a-1,b-2,c-2)\\=g(a-1,b-1,c-1)+g(a-2,b-2,c-1)
\end{split}
\end{equation}

 (iv) If $a-c=3l+2$, then
\begin{equation}\label{fact1-4}
\begin{split}
g(a,b,c)+g(a-3,b-3,c-2)+1=g(a-2,b-1,c)+g(a-1,b-2,c-2)\\=g(a-1,b-1,c-1)+g(a-2,b-2,c-1)+1
\end{split}
\end{equation}

 We first prove (\ref{recur1}) for the case of even $b$. We distinguish 6 sub-cases, corresponding to the values of $a-c\pmod{6}$.

If $a-c\equiv 0\pmod{6}$, then facts (i) and (iii) above allow us cancel out almost all of the factors 13 and 14 on the two sides of the equation (\ref{recur1}). After these cancelations, verifying  (\ref{recur1}) boils down to checking the simple equation
\begin{equation}\begin{split}
14h(a,b,c)h(a-3,b-3,c-2)=h(a-2,b-1,c)h(a-1,b-2,c-2)\\+h(a-1,b-1,c-1)h(a-2,b-2,c-1).
\end{split}
\end{equation}
However, this readily follows from the definition of the function $h(a,b,c)$, since in this case we have $h(a,b,c)=1$, $h(a-3,b-3,c-2)=1$, $h(a-2,b-1,c)=3$, $h(a-1,b-2,c-2)=3$,  $h(a-1,b-1,c-1)=1$ and $h(a-2,b-2,c-1)=5$.


If $a-c\equiv 1\pmod{6}$, then we argue similarly to the previous case, except we now use the facts (ii) and (iii). The (\ref{recur1}) reduces to
\begin{equation}\begin{split}
h(a,b,c)h(a-3,b-3,c-2)=h(a-2,b-1,c)h(a-1,b-2,c-2)\\+h(a-1,b-1,c-1)h(a-2,b-2,c-1).
\end{split}
\end{equation}
This follows by the definition of $h(a,b,c)$, as in this case we have $h(a,b,c)=3$, $h(a-3,b-3,c-2)=1$, $h(a-2,b-1,c)=1$, $h(a-1,b-2,c-2)=1$,  $h(a-1,b-1,c-1)=2$ and $h(a-2,b-2,c-1)=1$.

If $a-c\equiv 2\pmod{6}$, then facts (i) and (iv) allow us to simplify (\ref{recur1}) to
\begin{equation}\begin{split}
14h(a,b,c)h(a-3,b-3,c-2)=13h(a-2,b-1,c)h(a-1,b-2,c-2)\\+h(a-1,b-1,c-1)h(a-2,b-2,c-1).
\end{split}
\end{equation}
This follows since now we have $h(a,b,c)=1$, $h(a-3,b-3,c-2)=2$, $h(a-2,b-1,c)=1$, $h(a-1,b-2,c-2)=1$, $h(a-1,b-1,c-1)=5$ and $h(a-2,b-2,c-1)=3$.

 If $a-c\equiv 3\pmod{6}$, then we apply facts (ii) and (iii) to reduce (\ref{recur1}) to
\begin{equation}\begin{split}
h(a,b,c)h(a-3,b-3,c-2)=h(a-2,b-1,c)h(a-1,b-2,c-2)\\+h(a-1,b-1,c-1)h(a-2,b-2,c-1).
\end{split}
\end{equation}
This equation holds since $h(a,b,c)=1$, $h(a-3,b-3,c-2)=5$, $h(a-2,b-1,c)=2$, $h(a-1,b-2,c-2)=2$, $h(a-1,b-1,c-1)=1$ and $h(a-2,b-2,c-1)=1$.

 If $a-c\equiv 4\pmod{6}$, then facts (i) and (iii) allow us to simplify (\ref{recur1}) to
\begin{equation}\begin{split}
14h(a,b,c)h(a-3,b-3,c-2)=h(a-2,b-1,c)h(a-1,b-2,c-2)\\+h(a-1,b-1,c-1)h(a-2,b-2,c-1),
\end{split}
\end{equation}
that is true by the definition of the function $h(a,b,c)$ (in this case $h(a,b,c)=2$, $h(a-3,b-3,c-2)=1$, $h(a-2,b-1,c)=5$, $h(a-1,b-2,c-2)=5$, $h(a-1,b-1,c-1)=3$ and $h(a-2,b-2,c-1)=1$).

Finally, if $a-c\equiv 5\pmod{6}$, then, by facts (ii) and (iv), (\ref{recur1}) reduces to
\begin{equation}\begin{split}
h(a,b,c)h(a-3,b-3,c-2)=13h(a-2,b-1,c)h(a-1,b-2,c-2)\\+h(a-1,b-1,c-1)h(a-2,b-2,c-1).
\end{split}
\end{equation}
As in the previous cases, this follows by the definition of the function $h(a,b,c)$, since in this case we have $h(a,b,c)=5$, $h(a-3,b-3,c-2)=3$, $h(a-2,b-1,c)=1$, $h(a-1,b-2,c-2)=1$,  $h(a-1,b-1,c-1)=1$ and $h(a-2,b-2,c-1)=2$.

The remaining case of odd $b$ turns out to follow from the above calculations. Indeed,  for $i=1,2,\dotsc,6$, verification of the case of $b$ odd and $a-c \equiv i\pmod{6}$ turns out to involve precisely the same calculation as the verification of the case of even $b$ and $a-c \equiv 3+i\pmod{6}$ (this is so because both the value of $h(a,b,c)$ and the identities in facts (i)--(iv) above remain unchanged when we simultaneously change the parity of $b$ and increase by 3 the residue of $a-c$ modulo 6).
\end{proof}

\begin{lemma}\label{form3} Assume that $a$, $b$ and $c$ are integers, and let $d=2b-a-2c$ and  $e=b+d-a$.

$(${\rm a}$)$. We have
\begin{equation}\label{recur3}
\begin{split}
\phi(a,b,c)\phi(a-2,b-2,c)=\phi(a-1,b-1,c)^{2}\\+ \phi(a,b,c+1)\phi(a-2,b-2,c-1)
\end{split}
\end{equation}
and
\begin{equation}\label{recur3b}
\begin{split}
\psi(a,b,c)\psi(a-2,b-2,c)=\psi(a-1,b-1,c)^{2}\\+\psi(a,b,c+1)\psi(a-2,b-2,c-1).
\end{split}
\end{equation}

$(${\rm b}$)$. If we assume in addition that $c=0$, then
\begin{equation}\label{recur3c}
\begin{split}
\phi(a,b,c)\phi(a-2,b-2,c)=\phi(a-1,b-1,c)^{2}\\+\phi(a,b,c+1)\phi(e,d,1)
\end{split}
\end{equation}
and
\begin{equation}\label{recur3d}
\begin{split}
\psi(a,b,c)\psi(a-2,b-2,c)=\psi(a-1,b-1,c)^{2}\\+\psi(a,b,c+1)\psi(e,d,1).
\end{split}
\end{equation}
\end{lemma}

\begin{proof}
We prove below (\ref{recur3}) and (\ref{recur3c}); (\ref{recur3b}) and (\ref{recur3d}) follow in the same fashion.

$(${\rm a}$)$. Similarly to the proof of Theorem \ref{form1}, we have the following facts.

(i) If $a-b+c=2k$, then
\begin{equation}\begin{split}
f(a,b,c)+f(a-2,b-2,c)=2f(a-1,b-1,c)\\=f(a,b,c+1)+f(a-2,b-2,c-1)=2k^2
\end{split}
\end{equation}

(ii) If $a-b+c=2k+1$, then
\begin{equation}\begin{split}
f(a,b,c)+f(a-2,b-2,c)+1=2f(a-1,b-1,c)+1\\=f(a,b,c+1)+f(a-2,b-2,c-1)=2k^2+2k+1
\end{split}
\end{equation}

(iii) If $a-c=3l$ or $3l+2$,  then
\begin{equation}\begin{split}
g(a,b,c)+g(a-2,b-2,c)=2g(a-1,b-1,c)+1\\=g(a,b,c+1)+g(a-2,b-2,c-1)+1
\end{split}
\end{equation}

(iv) If $a-c=3l+1$, then
\begin{equation}\begin{split}
g(a,b,c)+g(a-2,b-2,c)=2g(a-1,b-1,c)\\=g(a,b,c+1)+g(a-2,b-2,c-1)
\end{split}
\end{equation}

By the argument at the end of the proof of Theorem \ref{form1}, it suffices to consider the case of even $b$. We again have $6$ subcases, depending on the value of $a-c\pmod{6}$.

If $a-c\equiv 0\pmod{6}$, then facts (i) and (iii) allow us to cancel out the common exponents of 13 and 14 on both sides of (\ref{recur3}). After these simplifications, (\ref{recur3}) becomes
\begin{equation}\begin{split}
13h(a,b,c)h(a-2,b-2,c)=h(a-1,b-1,c)^2\\+h(a,b,c+1)h(a-2,b-2,c-1).
\end{split}
\end{equation}
By the definition of the function $h(a,b,c)$, this is equivalent to the obvious equality
\begin{equation}\label{numerical}
13\cdot 1\cdot 2=1\cdot 1+5\cdot 5.
\end{equation}
For the remaining equivalence classes we only present the equations that (\ref{recur3}) simplifies to; they all amount to numerical identities analogous to the one above.

If $a-c\equiv 1\pmod{6}$, simplifying (\ref{recur3}) using facts (ii) and (iv), we get
\begin{equation}\begin{split}
h(a,b,c)h(a-2,b-2,c)=h(a-1,b-1,c)^2\\+14h(a,b,c+1)h(a-2,b-2,c-1)\\
\end{split}
\end{equation}

If $a-c\equiv 2\pmod{6}$, then (\ref{recur3}) is simplified by facts (i) and (iii) to
\begin{equation}\begin{split}
13h(a,b,c)h(a-2,b-2,c)=h(a-1,b-1,c)^2\\+h(a,b,c+1)h(a-2,b-2,c-1)
\end{split}
\end{equation}

If $a-c\equiv 3\pmod{6}$, then facts (ii) and (iii) transform (\ref{recur3})  to
\begin{equation}\begin{split}
13h(a,b,c)h(a-2,b-2,c)=h(a-1,b-1,c)^2\\+14h(a,b,c+1)h(a-2,b-2,c-1)
\end{split}
\end{equation}

If $a-c\equiv 4\pmod{6}$, then using facts (i) and (iv) (\ref{recur3}) becomes
\begin{equation}\begin{split}
h(a,b,c)h(a-2,b-2,c)=h(a-1,b-1,c)^2\\+h(a,b,c+1)h(a-2,b-2,c-1)
\end{split}
\end{equation}

If $a-c\equiv 5\pmod{6}$, then by facts (ii) and (iii)  (\ref{recur3}) is equivalent to
\begin{equation}\begin{split}
13h(a,b,c)h(a-2,b-2,c)=h(a-1,b-1,c)^2\\+14h(a,b,c+1)h(a-2,b-2,c-1)
\end{split}
\end{equation}
As stated above, all these equations are equivalent to numerical identities analogous to (\ref{numerical}).

$(${\rm b}$)$. For $c=0$, we have $d=2b-a$ and $e=3b-2a$. We claim that the following equation holds:
\begin{equation}\label{equal3}
\phi(a-2,b-2,-1)=\phi(3b-2a,2b-a,1).
\end{equation}

Indeed, using the definition of the functions $f$ and $g$, one can readily check that
\begin{equation}\begin{split}
f(a-2,b-2,-1)=f(3b-2a,2b-a,1)\\=\left\lfloor\frac{(a-b-1)^2}{4}\right\rfloor
\end{split}
\end{equation}
and
\begin{equation}
\begin{split}
g(a-2,b-2,-1)=g(3b-2a,2b-a,1)\\=b^2-ab+a-b+\left\lfloor\frac{(a-1)^2}{3}\right\rfloor
\end{split}
\end{equation}
Since $3(b-2)+(a-2)-(-1) \equiv 3(2b-a)+(3b-2a)-1\pmod{6}$, we get
$h(a-2,b-2,-1)=h(3b-2a,2b-a,1)$. Thus, $\phi(a-2,b-2,-1)=\phi(3b-2a,2b-a,1)$.

Now (\ref{recur3c}) follows by (\ref{equal3}) and the $c=0$ specialization of (\ref{recur3}).
\end{proof}

\begin{lemma}\label{form4} Let $a$, $b$ and $c$ be integers, and set $d=2a-b-2c$.

$(${\rm a}$)$. We have
\begin{equation}\label{recur4}
\begin{split}
\phi(a,b,c)\phi(a-2,b-3,c-2)=\phi(a-1,b-1,c)\phi(a-1,b-2,c-2)\\+\phi(a-2,b-2,c-1)\phi(a,b-1,c-1)
\end{split}
\end{equation}
and
\begin{equation}\label{recur4b}
\begin{split}
\psi(a,b,c)\psi(a-2,b-3,c-2)=\psi(a-1,b-1,c)\psi(a-1,b-2,c-2)\\+\psi(a-2,b-2,c-1)\psi(a,b-1,c-1)
\end{split}
\end{equation}

$(${\rm b}$)$. If $d=0$, then
\begin{equation}\label{recur4c}
\begin{split}
\phi(a,b,c)\phi(a-2,b-3,c-2)=\psi(c,b-1,a-1)\phi(a-1,b-2,c-2)\\+\phi(a-2,b-2,c-1)\phi(a,b-1,c-1)
\end{split}
\end{equation}
and
\begin{equation}\label{recur4d}
\begin{split}
\psi(a,b,c)\psi(a-2,b-3,c-2)=\phi(c,b-1,a-1)\psi(a-1,b-2,c-2)\\+\psi(a-2,b-2,c-1)\psi(a,b-1,c-1)
\end{split}
\end{equation}
\end{lemma}

\begin{proof}
We prove below (\ref{recur4}) and (\ref{recur4c}); the other two equalities follow analogously.

 $(${\rm a}$)$. This part can be proved similarly to Lemmas \ref{form1} and \ref{form3}(a). The following facts are readily verified.

(i) For any $a$, $b$ and $c$, we have
\begin{equation}\begin{split}
f(a,b,c)+f(a-2,b-3,c-2)=f(a-1,b-1,c)+f(a-1,b-2,c-2)\\=f(a-2,b-2,c-1)+f(a,b-1,c-1)
\end{split}
\end{equation}

(ii) If $a-c=3l$,  then
\begin{equation}\begin{split}
g(a,b,c)+g(a-2,b-3,c-2)=g(a-1,b-1,c)+g(a-1,b-2,c-2)+1\\=g(a-2,b-2,c-1)+g(a,b-1,c-1)+1
\end{split}
\end{equation}

(iii) If $a-c=3l+1$ or $3l+2$ , then
\begin{equation}\begin{split}
g(a,b,c)+g(a-2,b-3,c-2)=g(a-1,b-1,c)+g(a-1,b-2,c-2)\\=g(a-2,b-2,c-1)+g(a,b-1,c-1)
\end{split}
\end{equation}

As explained in the proofs of the previous two lemmas, it suffices to consider for the case of even $b$. The above three identities (i)--(iii) allow us to simplify (\ref{recur4}) to a simple equation involving only evaluations of the function $h$. The latter is readily checked by considering separately the possible residues of $a-c$ modulo 6.

$(${\rm b}$)$. If $d=0$, then $2b-a=2c$. We will check the following equality between $\phi$ and $\psi$
\begin{equation}\label{equal4}
\phi(a-1,b-1,c)=\psi(c,b-1,a-1).
\end{equation}

Using the definition of the functions $f$ and $g$, one obtains that
\begin{equation}
f(a-1,b-1,c)=f(c,b-1,a-1)=\left\lfloor\frac{(c+a-b)^2}{4}\right\rfloor,
\end{equation}
and
\begin{equation}
\begin{split}
g(a-1,b-1,c)=g(c,b-1,a-1)\\=(b-a)(b-c-1)+\left\lfloor\frac{(a-c-1)^2}{3}\right\rfloor.
\end{split}
\end{equation}
Moreover, note that the definition of the functions $h$ and $p$ implies that if $(3b+a-c) +(3b'+a'-c')\equiv 0\pmod{6}$, then $h(a,b,c)=p(a',b',c')$. Therefore,
$h(a-1,b-1,c)=p(c,b-1,a-1)$. This implies that
$\phi(a-1,b-1,c)=\psi(c,b-1,a-1)$.

By (\ref{equal4}) and the  $d=0$ specialization of (\ref{recur4}) we obtain (\ref{recur4c}).
\end{proof}

\section{Proof of Theorem \ref{main}}


We prove Theorem \ref{main} by induction on the perimeter of the contour $\mathcal{C}(a,b,c)$ we used in Section 3 to define the regions $D_{a,b,c}$ and $E_{a,b,c}$. Denote this perimeter by $\mathcal{P}(a,b,c)$. If $a> c+d$, then $\mathcal{P}(a,b,c)=a+b+c+d+e+f= a+b+c+(2b-a-2c)+(b+d-a)+(a-c-d)=4b-2c$. If on the other hand we have $a\leq c+d$, then $\mathcal{P}(a,b,c)=a+b+c+(2b-a-2c)+(b+d-a)+(c+d-a)=4b-2a+2d=4b-2a+4b-2a-4c=8b-4a-4c$. In particular, $\mathcal{P}(a,b,c)$ is always even.

The base cases for our induction are those triples $(a,b,c)$ of indices of regions $D_{a,b,c}$ and $E_{a,b,c}$ for which {\it at least one} of the following conditions hold:

\medskip
$(i)$ $\mathcal{P}(a,b,c)\leq14$

\medskip
$(ii)$ $b\leq4$

\medskip
$(iii)$ $c+d\leq2$

\medskip
By the triangle inequality we have $b+c+ d+e+f\geq a$, and thus $2a\leq \mathcal{P}(a,b,c)$. Similarly we have $2b,2c\leq \mathcal{P}(a,b,c)$. Therefore, if $\mathcal{P}(a,b,c)\leq 14$, then $a,b,c\leq 7$. Taking into account also the inequalities $b\geq2$, $2b-a-2c\geq0$ and $3b-2a-2c\geq0$ which the indices of the regions $D_{a,b,c}$ and $E_{a,b,c}$ satisfy (see Section 3), it is straightforward to see that there are only 20 contours $\mathcal{C}(a,b,c)$ of perimeter at most 14 which give rise to non-empty regions of $D$- and $E$-type, i.e., which also satisfy $b\geq2$, $2b-a-2c\geq0$ and $3b-2a-2c\geq0$ (indeed, this is immediate for instance by running three `for' loops in a computer algebra package such as Maple; 8 of these fall into the case $a>c+d$, and 12 into the case $a\leq c+d$).

For each of these 20 contours the statement of Theorem \ref{main} can be readily checked using for instance the computer package \texttt{vaxmacs} written by David Wilson\footnote{This software is available at  \texttt{http://dbwilson.com/vaxmacs/}}. The number of tilings of each of these regions is returned in less than one second, and one easily checks that it agrees with the number given by the formulas $\phi(a,b,c)$ and $\psi(a,b,c)$ in the statement of  Theorem \ref{main}.

The base cases corresponding to inequalities $(ii)$ and $(iii)$ above can be handled by the same way. Indeed, if $b\leq 4$, then, since $d=2b-a-2c\geq 0 $, we get $c\leq b\leq 4$. By the same token, we also have $a\leq 2b\leq 8$. There are now 43 triples $(a,b,c)$ with components in the above ranges for which the $D$- and $E$-type regions of Section 3 are defined. Just as for the base case $(i)$,
with the help of \texttt{vaxmacs} one readily checks that $\M(D_{a,b,c})=\phi(a,b,c)$ and $\M(E_{a,b,c})=\psi(a,b,c)$ for all triples in this base case as well.

Similarly, if $c+d\leq 2$, then $d=2b-a-2c\leq 2$, so $2b\leq a+2c+2\leq a+6\leq b+d+6\leq b+8$. Thus $b\leq 8$, so $a\leq b+d\leq 10$. As far as $c$ is concerned, the defining condition of base case $(iii)$ implies $c\leq2$. From among the triples of integers $(a,b,c)$ with $a$, $b$ and $c$ in the above ranges, and with $c+d=2b-a-c\leq 2$, there are only 10 for which the regions $D_{a,b,c}$ and $E_{a,b,c}$ are defined. Using again  \texttt{vaxmacs}, one checks that for each of them the equalities $\M(D_{a,b,c})=\phi(a,b,c)$ and $\M(E_{a,b,c})=\psi(a,b,c)$ are satisfied.


For the induction step, assume that the statement of Theorem \ref{main} is true for all regions $D_{a,b,c}$ and $E_{a,b,c}$ with perimeter less than $k$ (for some fixed $k\geq 16$). We need to show that the statement is true for all regions with perimeter $k$. Let $D_{a,b,c}$ and $E_{a,b,c}$ be two regions having perimeter $\mathcal{P}(a,b,c)=k$.



By the base cases corresponding to inequalities $(ii)$ and $(iii)$ above, {\it we may assume throughout the rest of the proof that for all the $D$- and $E$-type regions that occur, their $b$-parameter is at least $5$, and the sum of their $c$- and $d$-parameters is at least $3$.}



\textbf{Case I. $a\leq c+d$.}

 We claim that in this case we may assume $e\geq2$. Indeed, as $a\leq c+d$, we have $a\leq2b-a-c$. Since $d\geq 0$, we also have $2b-a-2c\geq 0$. By adding these two inequalities we obtain $3a+3c\leq 4b$. Then we have $e=b+d-a=3b-2a-2c=3b-\frac{2}{3}(3a+3c)\geq b/3\geq 5/3$. Since $e$ is an integer, this implies that $e\geq 2$.

\textit{Case I.1. $a \geq 2$.}

Note that since $c+d\geq 3$, at least one of $c$ and $d$ is greater than or equal to 2. We divide Case I.1 into the following four subcases.

\quad\emph{Case I.1.a. } $d\geq2$ and $c\geq1$.

It readily follows (most easily by using the formula for the perimeter that we gave at the beginning of this section) that the four numbers  $\mathcal{P}(a-2,b-2,c)$, $\mathcal{P}(a-1,b-1,c)$, $\mathcal{P}(a,b,c+1)$ and $\mathcal{P}(a-2,b-2,c-1)$ are less than $\mathcal{P}(a,b,c)$ by 8, 4, 4, and 4 units, respectively. Thus, by the induction hypothesis we have
\[\M(D_{a-2,b-2,c})=\phi(a-2,b-2,c),\  \M(D_{a-1,b-1,c})=\phi(a-1,b-1,c),\]
\[ \M(D_{a,b,c+1})=\phi(a,b,c+1),\ \M(D_{a-2,b-2,c-1})=\phi(a-2,b-2,c-1),\]
\[\M(E_{a-2,b-2,c})=\psi(a-2,b-2,c),\  \M(E_{a-1,b-1,c})=\psi(a-1,b-1,c),\]
\[ \M(E_{a,b,c+1})=\psi(a,b,c+1),\ \M(E_{a-2,b-2,c-1})=\psi(a-2,b-2,c-1).\]
From the recurrences  (\ref{eq3-3}) and (\ref{eq4-3}) in Lemma \ref{con3}(a) (which applies, since as pointed out above the $b$-parameter can always be assumed to be at least 5, and since $e\geq2$ in this case, as shown above), and (\ref{recur3}) and (\ref{recur3b}) in Lemma \ref{form3}(a), we get $\M(D_{a,b,c})=\phi(a,b,c)$ and $\M(E_{a,b,c})=\psi(a,b,c)$.

\quad\emph{Case I.1.b.} $d\geq2$ and $c=0$.

This case can be treated similarly to the previous case by using (\ref{eq3-4}), (\ref{eq4-4}) in Lemma \ref{con3}(b) (which has the same hypotheses as  Lemma \ref{con3}(a), verified in the previous subcase), and  (\ref{recur3c}) and (\ref{recur3d}) in Lemma \ref{form3}(b). The fact that the perimeter of the new regions is smaller than $\mathcal{P}(a,b,c)$  follows by the calculations in Case I.1.a for all the new regions except $D_{e,d,1}$ and $E_{e,d,1}$. The latter have sides of lengths $e,d,1,b-2,a-2,f+1$ (see Figure \ref{conden9}), so they have perimeter $\mathcal{P}(a,b,c)-2<\mathcal{P}(a,b,c)$.

\quad\emph{Case I.1.c. } $c\geq2$ and $d\geq1$.

This can be treated similarly to the previous subcases by using  (\ref{eq3-1}), (\ref{eq4-1}) in Lemma \ref{con4}(a) (which applies, as $a\geq2$, $b\geq5$ and $c\geq2$), and (\ref{recur4}) and (\ref{recur4b}) in Lemma \ref{form4}(a). One readily verifies that the perimeters of all the new regions in the recurrences are smaller than the perimeter $\mathcal{P}(a,b,c)$ of $D_{a,b,c}$ and $E_{a,b,c}$.

\quad\emph{Case I.1.d.} $c\geq2$ and $d=0$.

This can be treated similarly to the previous subcase by using the equations  (\ref{eq3-2}), (\ref{eq4-2}) in Lemma \ref{con4}(b) (which has the same hypotheses as Lemma \ref{con4}(a), already verified in the previous subcase), and (\ref{recur4c}) and (\ref{recur4d})  in Lemma \ref{form4}(b). All the regions in the recurrences have smaller perimeter than $D_{a,b,c}$ and $E_{a,b,c}$. Indeed, except for the regions $D_{c,b-1,a-1}$ and $E_{c,b-1,a-1}$, this follows from the previous subcase. These two exceptional regions have sides $c,b-1,a-1,f,e-1,1$ (see Figure \ref{conden10}), so their perimeter is $\mathcal{P}(a,b,c)-2<\mathcal{P}(a,b,c)$.

\emph{Case I.2. } $a\leq 1 $.

Since $c+d\geq3$, we have in this case that $f=c+d-a\geq 2$. Reflecting the region $D_{a,b,c}$  (resp. $E_{a,b,c}$) about the side of length $b$ of the contour, we get the region $E_{f,e,d}$ (resp. $D_{f,e,d}$).

If $e\leq 4$, then, since the $b$-parameter of these regions is $e$, the statement of Theorem \ref{main} holds by the base case $(ii)$. We may therefore assume that $e\geq 5$. With this assumption the regions $D_{f,e,d}$ and $E_{f,e,d}$ satisfy the condition in Case I.1 above. Indeed, we have $f\geq 2$ (corresponding to the condition $a\geq 2$ in the original region). Furthermore, we have $f\leq d+c \Leftrightarrow 2b-2a-c\leq (2b-a-2c)+c \Leftrightarrow -a\leq 0$ (corresponding to the condition $a\leq c+d$ in the original region). Thus we obtain $\M(E_{f,e,d})=\psi(f,e,d)$ and $\M(D_{f,e,d})=\phi(f,e,d)$.

Moreover, similarly to the proof of Lemma \ref{form4}(b), one can verify that
\begin{equation}
\phi(a,b,c)=\psi(f,e,d)=\psi(2b-2a-c, 3b-2a-2c,2b-a-2c)
\end{equation}
and
\begin{equation}
\psi(a,b,c)=\phi(f,e,d)=\phi(2b-2a-c,3b-2a-2c,2b-a-2c)
\end{equation}
(recall that $d=2b-a-2c$, $e=3b-2a-2c$ and $f=c+d-a=2b-2a-c$).
Therefore,
\begin{equation}
\M(D_{a,b,c})=\M(E_{f,e,d})=\psi(f,e,d)=\phi(a,b,c)
\end{equation}
and
\begin{equation}
\M(E_{a,b,c})=\M(D_{f,e,d})=\phi(f,e,d)=\psi(a,b,c).
\end{equation}

\medskip
\textbf{Case II. $a\geq c+d+1$.}

\quad\emph{Case II.1. } $e\geq d$.

If $c\geq 2$, similarly to Case I.1.a we get $\M(D_{a,b,c})=\phi(a,b,c)$ and $\M(E_{a,b,c})=\psi(a,b,c)$ by the induction hypothesis and the recurrences (\ref{eq1-1}) and (\ref{eq2-1}) in Lemma \ref{con1}, and  (\ref{recur1}) and (\ref{recur1b}) in Lemma \ref{form1}.






If $c\leq1$, then $d\geq2$ (as $c+d\geq3$), so $e\geq 2$.

For $c=1$, this case follows by using the recurrences (\ref{eq3-3}) and (\ref{eq4-3}) in Lemma \ref{con3}(a) (which applies, as $a\geq c+d+1\geq 4$, $b\geq5$, and $e\geq d\geq2$),  and (\ref{recur3}) and (\ref{recur3b})  in Lemma \ref{form3}(a), in the same fashion as in Case I.1.a. For $c=0$, we use the recurrences (\ref{eq3-4}) and (\ref{eq4-4}) in Lemma \ref{con3}(b), and (\ref{recur3c}) and (\ref{recur3d}) in Lemma \ref{form3}(b), and we argue in a way similar to the one employed in Case I.1.b. 

\quad\emph{Case II.2. } $e\leq d$.

Note that, using the expressions of $e$ and $f$ in terms of $a$, $b$ and $c$, the condition defining this subcase is equivalent to $3b-2a-2c\leq2b-a-2c$, which in turn is the same as $b\leq a$.

We reflect the regions $D_{a,b,c}$ and $E_{a,b,c}$ about the horizontal line passing through the westernmost vertex of their contour. The resulting regions are $D_{b,a,f}$ and $E_{b,a,f}$, respectively; note that the $d$-, $e$- and $f$-parameters of these regions are $e$, $d$ and $c$, respectively.

One can readily verify from the definition that we have
\begin{equation}
\phi(a,b,c)=\phi(b,a,f), \text{and}\ \psi(a,b,c)=\psi(b,a,f).
\end{equation}

If $a\leq 4$ or $f+e\leq 2$, then the statement of Theorem \ref{main} holds by the base case $(iii)$ (this is because the $b$-parameter of these regions is $a$, while the $c$- and $d$-parameters are $f$ and $e$, respectively). We may therefore assume that $a\geq5$ and $f+e\geq3$.

If $c\geq 1$, then the resulting regions $D_{b,a,f}$ and $E_{b,a,f}$ satisfy the conditions of Case II.1. Indeed, $d\geq e$ (corresponding to the condition $e\geq d$ in the original regions), and $b\geq f+e+1\Leftrightarrow b\geq (-2b+2a+c)+(3b-2a-2c)+1\Leftrightarrow 0\geq 1-c$ (corresponding to the condition $a\geq c+d+1$ in the original regions). Therefore, we have
\begin{equation}\label{finalcase1}
\M(D(a,b,c))=\M(D(b,c,f))=\phi(b,c,f)=\phi(a,b,c)
\end{equation}
and
\begin{equation}\label{finalcase2}
\M(E(a,b,c))=\M(E(b,c,f))=\psi(b,c,f)=\psi(a,b,c).
\end{equation}

On the other hand, if $c=0$, then the regions $D_{b,a,f}$ and $E_{b,a,f}$ satisfy the conditions of  Case I above. Indeed, $b\leq f+e\Leftrightarrow b= -2b+2a+c+3b-2a-2c\Leftrightarrow 0\leq c$ (corresponding to the condition $a\leq c+d$ in the original regions). Thus we again have equations (\ref{finalcase1}) and (\ref{finalcase2}). This completes the proof of Theorem \ref{main}.

\section{Concluding remarks}

We have seen in this paper how Kuo's graphical condensation method can be used to prove the exact enumeration of tilings of two families of regions on the lattice obtained from the triangular lattice by drawing in the altitudes in all the unit triangles. We showed how this implies Blum's conjecture on the number of tilings of the hexagonal dungeons, which was open since 1999. This illustrates the power of this method on a lattice different from the square and hexagonal lattices, the two lattices involved in the vast majority of the applications of graphical condensation in the literature.

We conclude by pointing out that the method used in the original proof of the Aztec dungeon theorem (i.e., using a certain local replacement rule to transform the problem into a weighted Aztec diamond enumeration problem, see \cite{Ciucu}) can be considered for hexagonal dungeons as well, but despite initial promise it does not seem to lead to a solution. It does lead however to some new families of regions that are natural to consider from this point of view, and which have the number of their tilings given by products of two perfect powers. We are planning to present them in a subsequent paper.

\thispagestyle{headings}


\begin{thebibliography}{10}

\bibitem{Ciucu}
M. Ciucu
\emph{Perfect matchings and perfect powers},
J. Algebraic Combin. \textbf{17} (2003)
335--375.

\bibitem{Elkies}
N. Elkies, G. Kuperberg, M.Larsen, and J. Propp
\emph{Alternating-sign matrices and domino tilings (Part I)}, J. Algebraic Combin. \textbf{1} (1992), 111-132.

\bibitem{Kuo}
E. H. Kuo
\emph{Applications of Graphical Condensation for Enumerating Matchings and Tilings},
Theoretical Computer Science \textbf{319} (2004),
29--57.

\bibitem{hybrid}
T. Lai
\emph{Enumeration of hybrid domino-lozenge tilings}, J. Combin. Theory Ser. A \textbf{122} (2014), 53--81.

\bibitem{Propp}
J. Propp,
\emph{Enumeration of matchings: Problems and progress},
New Perspectives in Geometric Combinatorics,  Cambridge University Press, 1999, 255--291.

\bibitem{Propp2}
J. Propp,
\emph{Generalized Domino-Shuffling},
Theoretical Computer Science \textbf{303} (2003),
267--301.

\bibitem{Zhang}
W. Yan and F. Zhang
\emph{Graphical Condensation for Enumerating of Perfect Matchings},
J. Combin. Theory Ser. A \textbf{110} (2005), 113--125.

\end{thebibliography}
\end{document}